\documentclass[12pt,reqno]{amsart}
\usepackage{fullpage}
\usepackage{amsmath}
\usepackage{times}
\usepackage{color}
\usepackage{textcomp}
\usepackage{esint}
\usepackage{graphicx}
\newtheorem{theorem}{Theorem}[section]
\newtheorem{lemma}[theorem]{Lemma}
\newtheorem{proposition}[theorem]{Proposition}
\newtheorem{corollary}[theorem]{Corollary}
\theoremstyle{definition}
\newtheorem{definition}[theorem]{Definition}

\newtheorem{remark}[theorem]{Remark}
\numberwithin{equation}{section}

\renewcommand{\mod}[1]{{\ifmmode\text{\rm\ (mod~$#1$)}\else\discretionary{}{}{\hbox{ }}\rm(mod~$#1$)\fi}}
\newcommand{\ep}{\varepsilon}
\newsymbol\nmid 232D

\newcommand{\C}{{\mathbb C}}

\newcommand{\E}{{\mathbb E}}

\newcommand{\R}{{\mathbb R}}

\usepackage[colorlinks=true, pdfstartview=FitH, linkcolor=blue, citecolor=blue, urlcolor=blue]{hyperref}
\DeclareMathAlphabet{\mathpzc}{OT1}{pzc}{m}{it}

\begin{document}


\title{Densities in certain three-way prime number races}
\author{Jiawei Lin and Greg Martin}
\address{Department of Mathematics \\ University of British Columbia \\ Room 121, 1984 Mathematics Road \\ Vancouver, BC, Canada \ V6T 1Z2}
\email{jiawei.lin@alumni.ubc.ca}
\email{gerg@math.ubc.ca}
\maketitle

\begin{abstract}
Let~$a_1$, $a_2$, and~$a_3$ be distinct reduced residues modulo~$q$ satisfying the congruences $a_1^2 \equiv a_2^2 \equiv a_3^2 \mod q$. We conditionally derive an asymptotic formula, with an error term that has a power savings in~$q$, for the logarithmic density of the set of real numbers~$x$ for which $\pi(x;q,a_1) > \pi(x;q,a_2) > \pi(x;q,a_3)$. The relationship among the~$a_i$ allows us to normalize the error terms for the $\pi(x;q,a_i)$ in an atypical way that creates mutual independence among their distributions, and also allows for a proof technique that uses only elementary tools from probability.
\end{abstract}

\section{Introduction}
A major topic in comparative prime number theory is the study of the prime number races among distinct reduced residues $a_1,\dots,a_r\mod q$. More precisely, one studies the set of positive real numbers~$x$ for which the inequalities $\pi(x;q,a_1) > \cdots > \pi(x;q,a_r)$ hold, where as usual $\pi(x;q,a)$ denotes the number of primes up to~$x$ that are congruent to~$a$ modulo~$q$. It was shown by Rubinstein and Sarnak~\cite{RS} that the logarithmic density of this set,
\begin{equation} \label{delta by pi}
\delta_{q;a_1,\cdots,a_r} = \lim_{x\to\infty} \frac1{\log x} \int\limits_{\substack{2\le t\le x \\ \pi(t;q,a_1) > \cdots > \pi(t;q,a_r)}} \frac {dt} t,
\end{equation}
exists and is strictly between~$0$ and~$1$, under two assumptions:
\begin{itemize}
\item GRH: the generalized Riemann hypothesis, asserting that all nontrivial zeros $\rho=\beta+i\gamma$ of Dirichlet $L$-functions $L(s,\chi)$ satisfy $\beta=\frac12$;
\item LI: a linear independence hypothesis, asserting that the multiset $\{ \gamma\ge0\colon$there exists $\chi\mod q$ such that $L(\frac12+i\gamma,\chi)=0\}$ is linearly independent over the rational numbers.
\end{itemize}
For a fixed number of contestants~$r$, the logarithmic densities $\delta_{q;a_1,\cdots,a_r}$ approach $\frac1{r!}$ uniformly as $q\to\infty$. Several authors (the articles~\cite{FiM,HL,lamother,lam} are most closely related to the present work) have given asymptotic formulas for the difference for various numbers of contestants (including results when the number of contestants can grow with~$q$), and others have provided generalizations to number fields, function fields, and elliptic curves, and to the counting functions of integers with a fixed number of prime factors in arithmetic progressions.

In this paper, we investigate a special class of three-way prime number races, where the residues involved satisfy the congruence
\begin{equation} \label{stw}
a_1^2 \equiv a_2^2 \equiv a_3^2 \mod q.
\end{equation}
We use elementary ideas from probability, and an approach involving an unusual normalization, to establish an asymptotic formula for the corresponding density $\delta_{q;a_1,a_2,a_3}$ with a very good error term. To state our theorem, we must first define some notation.

\begin{definition} \label{b+ def}
For any Dirichlet character $\chi \mod q$, define
\[
b_+(\chi) = \sum_{\substack {\gamma > 0 \\ L(1/2+i\gamma, \chi) = 0}} \frac 1 {\frac14+\gamma^2}
\quad\text{and}\quad
b(\chi)= \sum_{\substack {\gamma \in \mathbb{R} \\ L(1/2+i\gamma, \chi) = 0}} \frac 1 {\frac14+\gamma^2}.
\]
Notice that $b(\chi) = b_+(\chi)+b_+(\overline\chi)$ by the functional equation for Dirichlet $L$-functions, assuming that $L(\frac12,\chi)\ne 0$ (which is a consequence of LI).
\end{definition}

\begin{definition} \label{Hi def}
Define the following sets of characters\mod q:
\begin{align*}
H_0 &= \{\chi\mod q\colon \chi(a_1) = \chi(a_2) = \chi(a_3) \},\\
H_1 &= \{\chi\mod q\colon \chi(a_2) = \chi(a_3) =- \chi(a_1) \}, \\
H_2 &= \{\chi\mod q\colon \chi(a_1) = \chi(a_3) =- \chi(a_2) \}, \\
H_3 &= \{\chi\mod q\colon \chi(a_1) = \chi(a_2) =- \chi(a_3) \}.
\end{align*}
\end{definition}

\begin{remark} \label{Hi remark}
All these sets have the property that $\chi \in H_i$ if and only if $\overline{\chi} \in H_i$. It is easy to verify that~$H_0$ is a subgroup of the group of Dirichlet characters\mod q and that $H_1$, $H_2$, and $H_3$, if nonempty, are cosets of that subgroup.

Furthermore, under the assumption~\eqref{stw}, we show in Lemma~\ref{3way almost unanimous lemma} below that $H_0$ is an index-$4$ subgroup of the group of characters\mod q and that $H_1$, $H_2$, and $H_3$ are all its cosets, so that every character\mod q is in exactly one of $H_0$, $H_1$, $H_2$, or~$H_3$.
\end{remark}

\begin{definition} \label{V def}
With $b_+(\chi)$ as in Definition~\ref{b+ def}, for $i\in\{0,1,2,3\}$ define
\[
V(q) = 2 \sum_{\substack{\chi\mod q \\ \chi\ne\chi_0}} b_+(\chi), \qquad
V_i =  32 \sum_{\chi \in H_i} b_+(\chi), \qquad\text{and } \eta_i = \frac{V_i}{4V(q)} - 1.
\]
\end{definition}

With this notation in place, we may now state the main theorem of this paper.

\begin{theorem} \label{main theorem}
Assume GRH and LI. If~$a_1$, $a_2$, and~$a_3$ are distinct reduced residues modulo~$q$ satisfying $a_1^2\equiv a_2^2\equiv a_3^2\mod q$, then
\[
\delta_{q;a_1,a_2,a_3} = \frac1{2\pi} \arctan \frac{\sqrt{V_1V_2+V_1V_3+V_2V_3}}{V_2} + O_\ep(q^{-1/2+\ep}).
\]
Moreover, if~$a_1$, $a_2$, and~$a_3$ are all quadratic residues or all quadratic nonresidues\mod q, then the error term can be improved to $O(1/\phi(q)\log q)$.
\end{theorem}

\noindent As it happens, most of this paper is concerned with proving the second assertion (with the additional hypothesis on the quadratic nature of~$a_1$, $a_2$, and~$a_3$), after which we derive the first assertion (with its weaker errror term) from it.

Note that if~$V_1$, $V_2$, and~$V_3$ are all quite close to one another (as we shall show is the case), then the argument of $\arctan$ in Theorem~\ref{main theorem} is approximately $\sqrt3$, so that the main term is approximately $\frac16$ as expected.
Stating the theorem with this main term of a perhaps unforeseen shape allows the error term to remain quite small. However, we can derive a simpler asymptotic formula from this theorem if we are less concerned with the quality of the error term:

\begin{corollary} \label{main cor}
Assume GRH and LI. If $a_1$, $a_2$, and $a_3$ are distinct reduced residues modulo~$q$ satisfying $a_1^2\equiv a_2^2\equiv a_3^2\mod q$, then
\[
\delta_{q;a_1,a_2,a_3} = \frac16 + \frac{\eta_i+\eta_k}{8\pi\sqrt3} - \frac{\eta_j}{4\pi\sqrt3} + O\bigg( \frac{(\log \log q)^2}{(\log q)^2}\bigg).
\]
\end{corollary}

\noindent This version of the result recovers a special case of a theorem of Lamzouri~\cite{lam}, with a somewhat simpler proof; see the end of Section~\ref{finally done section} for the details of the comparison.

We have four motivations for presenting Theorem~\ref{main theorem} and its proof. First, the theorem has a better error term than has been recorded in the literature for any prime number race with three or more competitors; indeed, for such races, it is rare to see a savings of a power of~$q$ at all. Second, our proof of Theorem~\ref{main theorem} involves an unusual normalization (see Definition~\ref{E* def} below) of the error terms for the $\pi(x;q,a_i)$, one that allows us to treat the three error terms connected to this race as random variables that are in fact independent, which we hope might inspire similar constructions in other settings. Third, much of the recent progress on prime number races has invoked powerful machinery from probability; we wanted to give an application in this subject where more elementary methods suffice. Finally, we were motivated by generalizing the discussion of the second author from~\cite{Mar}, which essentially treats the two smallest cases $q=8$ and $q=12$ of Theorem~\ref{main theorem} numerically, but with a heuristic analysis that anticipates the methods herein.

That being said, methods from the current literature in comparative prime number theory are capable of treating much more general circumstances, and also, if viewed from a suitable perspective, of providing formulas with error terms nearly as strong as that of Theorem~\ref{main theorem}. See Section~\ref{discussion section} (and also the end of Section~\ref{almost unanimous section}) for further discussion about this wider context.

The rest of the paper is organized as follows. We quote results from the literature in Section~\ref{background section} concerning the limiting logarithmic distributions of error terms for prime counting functions and the random variables that model them. It is in Section~\ref{almost unanimous section} that we define the atypical normalization of these error terms that allows us to treat them independently, and calculate their variances. Using known facts about Bessel functions, we exhibit in Section~\ref{Bessel section} the characteristic function of our random variables and derive some power series representations of them. In Sections~\ref{zeroth deriv section} and~\ref{second deriv section} we establish pointwise bounds between these characteristic functions and the characteristic function of normal variables with the same mean and variance, as well as between the second derivatives of these characteristic functions. This information allows us to compare the density functions and eventually the probabilities themselves of these two types of random variables in Section~\ref{final assembly section}, at which point we prove Theorem~\ref{main theorem} and Corollary~\ref{main cor} in Section~\ref{finally done section}. We conclude with some discussion of the relationship between our results and existing results, and espouse a viewpoint on how such results should be conceived, in Section~\ref{discussion section}.

\section{Background information} \label{background section}

The foundation of the method we use has appeared many times, certainly stimulated in this generation by~\cite{RS}. It will be most convenient for us to quote several definitions and results from work by Fiorilli and the second author~\cite{FiM}, starting with the traditional normalization of the error term for prime counting functions in arithmetic progressions.

\begin{definition} \label{Exqa def}
For any reduced residue $a\mod q$, define
\[
E(x;q,a) = \frac {\phi(q)\pi(x; q, a) - \pi(x)} {\sqrt{x}/\log x},
\]
where $\pi(x)=\pi(x;1,1)$ as usual.
\end{definition}

The following explicit formula for $E(x;q,a)$ is~\cite[Lemma~2.1]{RS}, simplified slightly by the assumption of~GRH.

\begin{lemma} \label{Exqa lemma}
Assume GRH. For any reduced residue $a\mod q$
\[
E(x; q,a) = -c_q(a) + \sum_{\substack{\chi \mod q \\ \chi\ne\chi_0}}  \overline{\chi}(a) E(x,\chi) + o(1)
\]
as $x\to\infty$; here
\begin{equation} \label{cq def}
c_q(a) = -1 + \#\{b \mod q \colon b^2 \equiv a \mod q \}
\end{equation}
and, for any Dirichlet character $\chi$,
\[
E(x, \chi) = \sum_{\substack {\gamma \in \mathbb{R} \\ L(1/2+i\gamma, \chi) = 0}} \frac {x^{i\gamma}} {1/2+i\gamma}
\]
(which converges conditionally when interpreted as the limit of $\sum_{|\gamma|<T}$ as $T$ tends to infinity).
\end{lemma}

It is convenient to be able to interpret the distribution of values of $E(x;q,a)$ in terms of certain random variables.

\begin{definition} \label{Z and X def}
For any Dirichlet character $\chi \mod q$, define the random variable
\[
Z_\chi = \sum_{\substack{  \gamma > 0 \\ L(1/2 + i\gamma, \chi) = 0}} \frac {Z_\gamma} {\sqrt{\frac14+\gamma^2}}
\]
where the $Z_\gamma$ are independently uniformly distributed on the unit circle in $\mathbb{C}$. We also use the notation $X_\gamma = \Re Z_\gamma$ and $X_\chi=\Re Z_\chi$, so that the $X_\gamma$ also form an independent collection of random variables, as do the $X_\chi$ assuming that the $L(s,\chi)$ have no zeros in common (which is a consequence of~LI).
\end{definition}

It is known that vector-valued relatives of $E(x;q,a)$ have limiting logarithmic distributions that can be expressed in terms of these random variables; the following proposition is~\cite[Proposition~2.3]{FiM}.

\begin{proposition} \label{general E to Z prop}
Assume LI. Let $\{c_\chi \colon \chi \mod q\}$ be a collection of $\C^r$-vectors, indexed by the Dirichlet characters\mod q, satisfying $c_{\overline{\chi}} = \overline{c_\chi}$. The limiting logarithmic distribution of any $\R^r$-valued function of the form
\[
\sum_{\chi \mod q} c_\chi E(x, \chi) + o(1)
\]
is the same as the distribution of the random variable 
\[
2\Re \sum_{\chi \mod q} c_\chi Z_\chi.
\]
\end{proposition}

\section{Special three-way races and error terms with atypical normalizations} \label{almost unanimous section}

In this section we set out some notation that will be used throughout the main part of this paper (from this point through Section~\ref{finally done section}). In particular, the assumptions on~$a_1$, $a_2$, and~$a_3$ in the first definition will be in force in these sections without explicit mention, as our main result is concerned only with these special three-way prime number races.

\begin{definition} \label{a def}
Let~$a_1$, $a_2$, and~$a_3$ denote distinct reduced residues\mod q such that
\[
a_1^2 \equiv a_2^2 \equiv a_3^2\mod q.
\]
We assume, through the middle of Section~\ref{finally done section}, that $a_1$, $a_2$, and $a_3$ are either all quadratic residues or all quadratic non\-residues\mod q. (Later in Section~\ref{finally done section} we will discuss how this assumption can be removed to establish Theorem~\ref{main theorem} in its entirety).

We will often use $i,j,k$ as indices that denote a generic permutation $(i,j,k)$ of $(1,2,3)$. For example, we define
\[
a_0 \equiv a_ia_ja_k^{-1}\mod q,
\]
which is independent of the permutation $(i,j,k)$.
\end{definition}

\begin{remark}
It is easy to show that most integers~$q$ possess three distinct reduced residues~$a_1$, $a_2$, and~$a_3$ such that the congruences $a_1^2\equiv a_2^2\equiv a_3^2\mod q$ are satisfied---indeed, the integers that do not are precisely the integers with primitive roots. It is also straightforward to show that one almost always can choose these reduced residues so that $a_1$, $a_2$, and $a_3$ are all quadratic nonresidues; for example, such a choice is possible whenever~$q$ has at least three distinct odd prime factors.
\end{remark}

In the special situation described in Definition~\ref{a def}, the sets $H_i$ from Definition~\ref{Hi def} have a tidy relationship with one another.

\begin{lemma} \label{3way almost unanimous lemma}
The sets $H_0$, $H_1$, $H_2$, and $H_3$ partition the group of Dirichlet characters\mod q into four subsets each of cardinality $\frac14\phi(q)$.
\end{lemma}

\begin{proof}
The assumption $a_1^2 \equiv a_2^2 \equiv a_3^2 \mod q$ implies that $\chi(a_1^2) = \chi(a_2^2) = \chi(a_3^2)$, or equivalently $\chi^2(a_1) = \chi^2(a_2) = \chi^2(a_3)$, for every Dirichlet character $\chi\mod q$. Therefore if $a_1^2 \equiv a_2^2 \equiv a_3^2 \mod q$, then for every $\chi\mod q$ each of $\chi(a_1)$, $\chi(a_2)$, and $\chi(a_3)$ must be a square root of the common value $\chi^2(a_i)$. Since there are only two such square roots, at least two of the character values must be equal, and the third value (if not equal to the other two) is the negative of the others. In particular, the sets $H_0$, $H_1$, $H_2$, and $H_3$ partition the group of characters\mod q. The fact that they have equal cardinalities, which must necessarily be $\frac14\phi(q)$, now follows from the observation made in Remark~\ref{Hi remark} that $H_1$, $H_2$, and $H_3$ are all cosets of the subgroup~$H_0$.
\end{proof}

We are also able to simplify certain combinations of character values in this special situation.

\begin{lemma} \label{chi cancel lemma}
For any permutation $(i,j,k)$ of $(1,2,3)$,
\[
\overline{\chi}(a_i)+\overline{\chi}(a_0)-\overline{\chi}(a_j)-\overline{\chi}(a_k) = \begin{cases}
4 \overline{\chi}(a_i), &\text{if } \chi\in H_i, \\
0, &\text{otherwise},
\end{cases}
\]
where $H_i$ is the set of characters from Definition~\ref{Hi def}.
\end{lemma}

\begin{proof}
It is immediate from Definitions~\ref{Hi def} and~\ref{a def} that $\chi(a_0) = \chi(a_i)$ if $\chi\in H_0\cup H_i$ and that $\chi(a_0) = -\chi(a_i)$ if $\chi\in H_j\cup H_k$, and then that
\begin{align*}
\overline{\chi}(a_i)+\overline{\chi}(a_0)-\overline{\chi}(a_j)-\overline{\chi}(a_k) &= \begin{cases}
\overline{\chi}(a_i)+\overline{\chi}(a_i)-\overline{\chi}(a_i)-\overline{\chi}(a_i) = 0, &\text{if } \chi\in H_0, \\
\overline{\chi}(a_i)+\overline{\chi}(a_i)+\overline{\chi}(a_i)+\overline{\chi}(a_i) = 4\overline{\chi}(a_i), &\text{if } \chi\in H_i, \\
\overline{\chi}(a_i)-\overline{\chi}(a_i)+\overline{\chi}(a_i)-\overline{\chi}(a_i) = 0, &\text{if } \chi\in H_j, \\
\overline{\chi}(a_i)-\overline{\chi}(a_i)-\overline{\chi}(a_i)+\overline{\chi}(a_i) = 0, &\text{if } \chi\in H_k.
\end{cases}
\end{align*}
\end{proof}

At this point we introduce an unusual normalization, tailored to this special situation, of the error terms $E(x;q,a)$ from Definition~\ref{Exqa def}.

\begin{definition} \label{E* def}
For any permutation $(i,j,k)$ of $(1,2,3)$, define
\[
E^{*}(x; q, a_i) = \alpha + E(x; q, a_i) + E(x; q, a_0) - E(x; q, a_j) - E(x; q, a_k),
\]
where $\alpha=c_q(a_0) - c_q(a_i)$ (which, by assumption, is independent of $i\in\{1,2,3\}$).
Note that
\[
E(x; q, a_i) - \tfrac12E^{*}(x; q, a_i) = \tfrac12\big( E(x; q, a_i) + E(x; q, a_j) + E(x; q, a_k) - E(x; q, a_0) - \alpha \big)
\]
is independent of the permutation $(i,j,k)$, so that the ordering of the $E^{*}$ terms is always the same as the ordering of the $E$ terms. In particular, $E^*(x; q, a_i) > E^*(x; q, a_j) > E^*(x; q, a_k)$ if and only if $\pi(x; q, a_i) > \pi(x; q, a_j) > \pi(x; q, a_k)$, so that equation~\eqref{delta by pi} becomes
\begin{equation} \label{delta by E^*}
\delta_{q;a_1,a_2,a_3} = \lim_{x\to\infty} \frac1{\log x} \int\limits_{\substack{2\le t\le x \\ E^*(t;q,a_1) > E^*(t;q,a_2) > E^*(t;q,a_3)}} \frac {dt} t.
\end{equation}
\end{definition}

We can immediately start to see the benefit of this atypical normalization, in that the explicit formulas for the $E(x; q, a_i)$ involve disjoint sets of Dirichlet characters.

\begin{lemma} \label{E* decomp lemma}
Assume GRH. For $i\in\{1,2,3\}$, we have $
E^{*}(x; q, a_i) = 4 \sum_{\chi \in H_i} \overline{\chi}(a_i)E(x,\chi) + o(1).$
\end{lemma}

\begin{proof}
By Lemma~\ref{Exqa lemma},
\begin{align}
E^{*}(x; q, a_i) &= \big( \alpha + E(x; q, a_0) \big) + E(x; q, a_i) + - E(x; q, a_j) - E(x; q, a_k) \notag \\
&= -c_q(a_i)- c_q(a_i) + c_q(a_j)+c_q(a_k) \label{four quantities} \\
&\qquad{}+ \sum_{\substack{\chi \mod q \\ \chi\ne\chi_0}} \big( \overline{\chi}(a_i)+\overline{\chi}(a_0)-\overline{\chi}(a_j)-\overline{\chi}(a_k) \big) E(x,\chi) + o(1). \notag
\end{align}
The assumption that $a_1$, $a_2$, and $a_3$ are either all quadratic residues or all quadratic non\-residues\mod q means that the four quantities on line~\eqref{four quantities} are all equal and thus cancel one another. The lemma now follows from Lemma~\ref{chi cancel lemma}.
\end{proof}

We remark that Lemma~\ref{E* decomp lemma} is the only place in our argument where we use the standing assumption that $a_1$, $a_2$, and $a_3$ are either all quadratic residues or all quadratic non\-residues\mod q. In Section~\ref{finally done section} we show how we can derive the general form of Theorem~\ref{main theorem} from the version that requires this assumption.

At this point we are ready to introduce certain random variables that model, in their distributions, the normalized error terms $E^{*}(x; q, a_i)$.

\begin{definition} \label{Xi def}
For $i\in\{1,2,3\}$, define the random variable 
\[
X_i = 8\sum_{\chi \in H_i} X_\chi,
\]
where $X_\chi$ is as in Definition~\ref{Z and X def}; note that the disjointness of $H_1$, $H_2$, and $H_3$ and the independence of the $X_\chi$ imply that $X_1$, $X_2$, and $X_3$ are mutually independent.
Further, for any permutation $(i,j,k)$ of $(1,2,3)$, define the vector-valued random variable
\[
X_{i,j,k} = (X_i,X_j,X_k).
\]
\end{definition} 

\begin{lemma} \label{X variance lemma}
For $i\in\{1,2,3\}$, the variance of~$X_i$ is the quantity~$V_i$ from Definition~\ref{V def}.
\end{lemma}

\begin{proof}
By Definitions~\ref{Xi def} and~\ref{Z and X def},
\begin{align*}
\sigma^2(X_i) &= \sigma^2 \bigg( 8\sum_{\chi \in H_i} X_\chi \bigg) = 64  \sigma^2 \bigg( \sum_{\chi \in H_i} \sum_{\substack{  \gamma > 0 \\ L(1/2 + i\gamma, \chi) = 0}} \frac {X_\gamma} {\sqrt{\frac14+\gamma^2}} \bigg).
\end{align*}
Since the $X_\gamma$ are independent by assumption,
\begin{align*}
\sigma^2(X_i) &= 64 \sum_{\chi \in H_i} \sum_{\substack{  \gamma > 0 \\ L(1/2 + i\gamma, \chi) = 0}} \frac {\sigma^2(X_\gamma)} {{\frac14+\gamma^2}} 
= 32 \sum_{\chi \in H_i} \sum_{\substack{  \gamma > 0 \\ L(1/2 + i\gamma, \chi) = 0}} \frac {1} {{\frac14+\gamma^2}}
= 32 \sum_{\chi \in H_i} b_+(\chi) 
= V_i
\end{align*}
as claimed.
\end{proof}

Our final proposition of the section records the fact that these random variables truly are substitute objects of study for the normalized prime-counting error terms. Recall the logarithmic density $\delta_{q;a_i,a_j,a_k}$ from Definition~\ref{delta by E^*}.

\begin{proposition} \label{E is X prop}
Assume GRH and LI. For any permutation $(i,j,k)$ of $(1,2,3)$, the limiting logarithmic distribution of the vector-valued function $\big( E^{*}(x; q, a_i), E^{*}(x; q, a_j), E^{*}(x; q, a_k) \big)$ is the same as the distribution of the random variable $X_{i,j,k}$; in particular,
\[
\delta_{q;a_1,a_2,a_3} = \Pr(X_1>X_2>X_3).
\]
\end{proposition}

\begin{proof}
By Lemma~\ref{E* decomp lemma},
\begin{align*}
\big( & E^{*}(x; q, a_i), E^{*}(x; q, a_j), E^{*}(x; q, a_k) \big) \\
&= 4 \bigg(\sum_{\chi \in H_i} (\overline\chi(a_i),0,0)E(x,\chi) + \sum_{\chi \in H_j} (0,\overline\chi(a_j),0)E(x,\chi) + \sum_{\chi \in H_k} (0,0,\overline\chi(a_k))E(x,\chi) \bigg) + o(1),
\end{align*}
whose limiting logarithmic distribution, by Proposition~\ref{general E to Z prop}, is the same as the distribution of
\[
2\Re \bigg( 4 \bigg(\sum_{\chi \in H_i} (\overline\chi(a_i),0,0)Z_\chi + \sum_{\chi \in H_j} (0,\overline\chi(a_j),0)Z_\chi + \sum_{\chi \in H_k} (0,0,\overline\chi(a_k))Z_\chi \bigg) \bigg).
\]
Since $Z_\chi$ is uniformly distributed on the unit circle in~$\C$ and $\overline\chi(a_i)$ is a point on the unit circle, we have simply $\overline\chi(a_i)Z_\chi = Z_\chi$, and similarly with~$i$ replaced by~$j$ or~$k$. Therefore
\begin{align*}
2\Re & \bigg( 4 \bigg(\sum_{\chi \in H_i} (\overline\chi(a_i),0,0)Z_\chi + \sum_{\chi \in H_j} (0,\overline\chi(a_j),0)Z_\chi + \sum_{\chi \in H_k} (0,0,\overline\chi(a_k))Z_\chi \bigg) \bigg) \\
&= 8 \Re \bigg(\sum_{\chi \in H_i} (Z_\chi,0,0) + \sum_{\chi \in H_j} (0,Z_\chi,0) + \sum_{\chi \in H_k} (0,0,Z_\chi) \bigg) \\
&= 8 \Re \bigg(\sum_{\chi \in H_i}Z_\chi, \sum_{\chi \in H_j} Z_\chi,  \sum_{\chi \in H_k} Z_\chi \bigg) = (X_i, X_j, X_k)
\end{align*}
as claimed. The final assertion follows from equation~\eqref{delta by E^*} and the definition of a limiting logarithmic distribution.
\end{proof}

\begin{remark}
From Definition~\ref{Xi def}, we note that we only need to assume GRH and LI for Dirichlet $L$-functions corresponding to the subset $H_1\cup H_2\cup H_3$ of Dirichlet characters\mod q.
\end{remark}

We have shown that the standing assumptions from Definition~\ref{a def} imply that atypical normalizations of the error terms for prime counting functions can be made independent of one another. Our proof used the fact that the sets from Definition~\ref{Hi def} comprised all Dirichlet characters\mod q; said another way, every $\chi\mod q$ takes at most two distinct values on~$a_1$, $a_2$, and~$a_3$. It turns out that a property of this type is more fundamental to our method than the original congruence assumption, an idea which we now take a slight detour to explore.

\begin{definition} \label{AU def}
Given distinct reduced residues $a_1,\dots,a_r\mod q$, we call a Dirichlet character $\chi\mod q$ {\em almost unanimous on $\{a_1,\dots,a_r\}$} if there exists an index $1\le k_\chi\le r$ and a complex number $\omega_\chi$ such that $\chi(a_1) = \cdots = \chi(a_{k_\chi-1}) = \chi(a_{k_\chi+1}) = \cdots = \chi(a_r) = \omega_\chi$.

Note that this definition includes the possibility that $\chi(a_{k_\chi})$ is also equal to $\omega_\chi$ (in which case $k_\chi$ can take any value in $\{1,\dots,r\}$); for example, the principal character~$\chi_0$ is always almost unanimous on any set of reduced residues. Note also that if $\chi$ is almost unanimous on $\{a_1,\dots,a_r\}$, then so is $\overline\chi$, and $k_{\overline\chi} = k_\chi$ and $\omega_{\overline\chi} = \overline{\omega_\chi}$.

Further, we call the set $\{a_1,\dots,a_r\}$ itself {\em almost unanimous} if every Dirichlet character\mod q is almost unanimous on $\{a_1,\dots,a_r\}$.
\end{definition}

If a set $\{a_1,\dots,a_r\}$ of distinct reduced residues is almost unanimous, then one can create an atypical normalization by subtracting the quantity $\sum_{\chi\mod q} \omega_\chi E(x,\chi)$ from each error term $E(x;q,a)$. Note that this quantity is real-valued since $\omega_{\overline\chi} = \overline{\omega_\chi}$, and therefore subtracting it from every $E(x;q,a)$ is order-preserving. The resulting differences have the form
\[
E(x;q,a_j) - \sum_{\chi\mod q} \omega_\chi E(x,\chi) = \sum_{\substack{\chi\mod q \\ k_\chi = j}} (\overline\chi(a) - \omega_\chi) E(x,\chi);
\]
in particular, the sets of characters appearing in the sums for different values of~$j$ are disjoint. As a result, assuming GRH and LI, the random variables modeling this atypically normalized error term will be mutually independent.

An examination of Definition~\ref{E* def} reveals that the quantities $E^*(x;q,a_i)$, up to constant factors, are precisely the result of applying this construction (the supplemental residue~$a_0$, while making the definition concise, is not crucial to the construction). In principle, then, this process of atypically normalizing the error terms for any almost unanimous set of residues would result in independent error terms.

The unfortunate news, however, is that there are no almost unanimous sets of $r\ge3$ residues other than the ones described in Definition~\ref{a def}. (All sets of $1$ or $2$ residue classes are trivially almost unanimous, and there are many ways to normalize two of these error terms to create independent functions---including simply replacing $E(x;q,a)$ and $E(x;q,b)$ with $E(x;q,a)-E(x;q,b)$ and~$0$, which is common practice.) The following two lemmas justify this anticlimactic assertion.

\begin{lemma}
If $r\ge4$, then there does not exist an almost unanimous set $a_1,\dots,a_r$ of distinct reduced residues\mod q.
\end{lemma}

\begin{proof}
Suppose, for the sake of contradiction, that $a_1,\dots,a_r$ is almost unanimous.
Since $a_1\not\equiv a_2\mod q$, there exists a character~$\chi_1$ such that $\chi_1(a_1) \ne \chi_1(a_2)$. By assumption, $\chi_1$ is almost unanimous on $\{a_1,\dots,a_r\}$; without loss of generality, $\chi_1(a_2) = \cdots = \chi_1(a_r)$. Similarly, there exists a character~$\chi_2$ such that $\chi_1(a_{r-1}) \ne \chi_1(a_r)$, and without loss of generality, $\chi_2(a_1) = \cdots = \chi_2(a_{r-1})$. (It is in this second ``without loss of generality'' step that we use the assumption $r\ge4$, so that there is no overlap between $\{a_1,a_2\}$ and $\{a_{r-1},a_r\}$.)

Now set $\chi_3=\chi_1\chi_2$. We see immediately that $\chi_3(a_2) = \cdots = \chi_3(a_{r-1})$. However, the known facts $\chi_1(a_1) \ne \chi_1(a_2)$ and $\chi_2(a_1) = \chi_2(a_2)$ imply that $\chi_3(a_1) \ne \chi_3(a_2)$; similarly, $\chi_1(a_{r-1}) = \chi_1(a_r)$ and $\chi_2(a_{r-1}) \ne \chi_2(a_r)$ imply that $\chi_3(a_{r-1}) \ne \chi_3(a_r)$. It follows that $\chi_3$ is not almost unanimous on $\{a_1,\dots,a_r\}$, contrary to assumption.
\end{proof}

\begin{lemma} \label{l5.1}
Let $a_1, a_2, a_3$ be distinct reduced residue classes\mod q. Then $\{a_1,a_2,a_3\}$ is almost unanimous\mod q if and only if $a_1^2 \equiv a_2^2 \equiv a_3^2 \mod q$.
\end{lemma} 

\begin{proof}
The proof of Lemma~\ref{3way almost unanimous lemma} shows that if $a_1^2 \equiv a_2^2 \equiv a_3^2 \mod q$ then $\{a_1,a_2,a_3\}$ is almost unanimous\mod q.

Conversely, suppose that $\{a_1,a_2,a_3\}$ is almost unanimous\mod q. Define
\begin{align*}
G_1 &= \{ \chi\mod q\colon \chi(a_2)=\chi(a_3)\} \\
G_2 &= \{ \chi\mod q\colon \chi(a_1)=\chi(a_3)\} \\
G_3 &= \{ \chi\mod q\colon \chi(a_1)=\chi(a_2)\},
\end{align*}
and note that $G_1\cap G_2\cap G_3 = H_0$ as in Definition~\ref{Hi def}; moreover, by the definition of almost unanimous, $G = G_1\cup G_2\cup G_3$. It is obvious that each~$G_i$ is a subgroup of the group~$G$ of Dirichlet characters\mod q. Furthermore, for any pair of distinct residues\mod q, there is always a Dirichlet character\mod q that takes different values on the two residues; in particular, the $G_i$ are proper subgroups.

Scorza (see~\cite{Z}) proved that a group~$G$ is the union of three proper subgroups~$G_1$, $G_2$, and~$G_3$ if and only if it has a quotient isomorphic to the Klein $4$-group~$K$ (a result that has been rediscovered more than once---see~\cite{HR} for example), in which case $G_1$, $G_2$, and $G_3$ are the inverse images of the three two-element subgroups of~$K$. In particular, the square of every element of~$G$ is in $G_1\cap G_2\cap G_3 = H_0$. In our situation, we deduce that $\chi^2(a_1) = \chi^2(a_2) = \chi^2(a_3)$ for every $\chi\mod q$, which implies that $\chi(a_1^2) = \chi(a_2^2) = \chi(a_3^2)$ for every $\chi\mod q$; this situation is possible only if $a_1^2 \equiv a_2^2 \equiv a_3^2 \mod q$.
\end{proof}

While there seems to be no direct generalization of our construction, we hope that the ideas described herein might inspire other beneficial atypical normalizations in the future.

\section{Bessel functions and bounds for characteristic functions} \label{Bessel section}

In this section we exhibit exact formulas for the characteristic function of the random variables~$X_i$ introduced in the previous section, as well as various estimates and series representations of those characteristic functions that will be needed in our analysis.

\begin{definition} \label{Bessel def}
Let $J_0(z)$ be the standard Bessel function of order~$0$. Let $\lambda_n$ be the coefficients in the power series expansion
\[
\log J_0(z) = \sum_{n=0}^\infty \lambda_n z^n,
\] 
which is valid for $|z| \le \frac{12}5$ since~$J_0$ has no zeros in this disk: this assertion can be verified computationally for real $z$---see the graph of~$J_0(x)$ in Figure~\ref{Bessel figure}---while Hurwitz~\cite{Hur} proved that $J_0(z)$ has no nonreal zeros (see also~\cite{HS}).
\end{definition}

The following lemma is~\cite[Lemma~2.8]{FiM}:

\begin{lemma} \label{lambda lemma}
With $\lambda_n$ as in Definition~\ref{Bessel def}:
\begin{enumerate}
\item $\lambda_n \ll \big(\frac 5 {12} \big)^n$ uniformly for $n \ge 0$;
\item $\lambda_0 = 0$ and $\lambda_{2m-1} = 0$ for every $m \ge 1$;
\item $\lambda_{2m} < 0$ for every $m \ge 1$.
\end{enumerate}
\end{lemma}

It would be advantageous if this Bessel function were decreasing for $x\ge 0$, say, so that we could bound the tail of the function $|J_0(x)|$ simply by $|J_0(\kappa)|$ for any fixed $0 \le \kappa \le x$. Inconveniently, $J_0(x)$ and its derivatives have oscillations in sign; it is the case, however, that their values are contained in a gradually decaying envelope. Consequently, their values near $x=0$ are indeed their largest, an observation we codify in the following lemma.

\begin{definition} \label{K def}
Define $K(x) = -J'_0(x)/x$, with the value $K(0) = \frac12$ chosen for continuity. Also, define $D(x) = -J'_0(x)/J_0(x) = xK(x)/J_0(x)$.
\end{definition}

\begin{figure}[hbt]
\includegraphics[height=6cm]{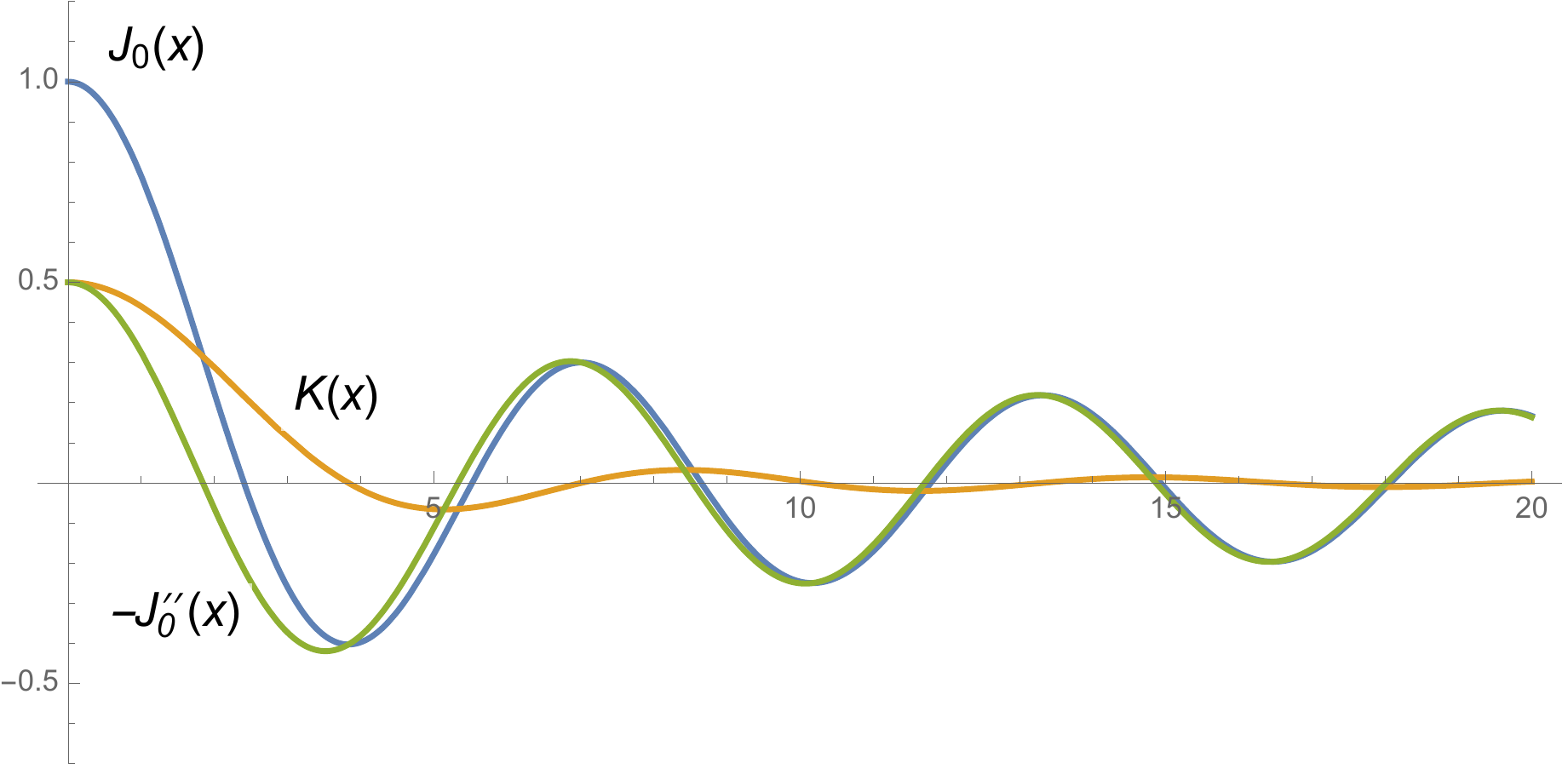}
\caption{The Bessel function and its relatives}
\label{Bessel figure}
\end{figure}

\begin{lemma} \label{Bessel sorta decreasing lemma}
If $0 \le \kappa \le \frac12$, then $J_0(\kappa) $, $K(\kappa)$, $-J''(\kappa)$, and $D(\kappa)$ are all positive; and for all real numbers $x$ with $|x| \ge \kappa$ we have $J_0(\kappa) \ge |J_0(x)|$ and $K(\kappa) \ge |K(x)|$ and $|J_0''(\kappa)| \ge |J_0''(x)|$. In particular,
\begin{equation} \label{Bessel derivative bounds}
|K(x)| \le \frac12 \quad\text{and}\quad |J_0''(x)| \le \frac12 \quad\text{for all } x\in\R.
\end{equation}
\end{lemma}

\begin{proof}
These assertions are clear from the graphs of the functions in Figure~\ref{Bessel figure} (since the first three functions are even, we may restrict attention to $x\ge0$); a rigorous proof is unenlightening, and we omit most of the details. Derivatives of Bessel functions are related to Bessel functions of higher order, and in particular
\begin{align*}
K(x) = \frac{J_1(x)}x \quad\text{and}\quad
-J_0''(x) = J_0(x) - \frac{J_1(x)}x = \frac{J_0(x)-J_2(x)}2.
\end{align*}
Serviceable bounds for these functions can be easily derived from~\cite[Section~VII.3, equation~(1)]{Wat} and the prior equations, showing that the lemma is true for $x\ge6$, say. The computations establishing the lemma for the remaining range can be done to any desired accuracy by computer. The smallest value of the three functions for $\kappa\in[0,\frac12]$ is $-J_0''(\frac12) > 0.45$, while the closest any of these functions come to violating the asserted inequality is the local minimum of $-J''_0(x)$ near $x=3.5$, at which $-J_0''(x) > -0.42$.
\end{proof}

The following convergent infinite products of Bessel functions is central in the subject of prime number races.

\begin{definition} \label{F Phi def}
For any Dirichlet character $\chi$, define
\[
F(z, \chi) = \prod_{\substack{  \gamma > 0 \\ L(\frac 1 2 + i\gamma, \chi) = 0}} J_0 \bigg( \frac {2z} {\sqrt{\frac14+\gamma^2}} \bigg).
\]
Then define, for any permutation $(i,j,k)$ of $(1,2,3)$,
\[
\Phi_i(z) = \prod_{\chi \mod q} F\big(|\chi(a_i)+\chi(a_0)-\chi(a_j)-\chi(a_k)|z, \chi \big) = \prod_{\chi \in H_i} F(4z,\chi),
\]
where the last equality holds by Lemma~\ref{chi cancel lemma}.
The products defining $F(z, \chi)$ and $\Phi_i(z)$ converge uniformly on bounded subsets of the complex plane (a fact that will follow from the upper bounds we establish below for these functions).
\end{definition}

We can immediately see the relevance of this function to the characteristic functions $\hat X_i(z)$ of the random variables~$X_i$ from Definition~\ref{Xi def}.

\begin{proposition}
For $i\in\{1,2,3\}$, we have $\hat X_i(z)=\Phi_i(z)$, where $\Phi_i(z)$ is as in Definition~\ref{F Phi def}.
\end{proposition}

\begin{proof}
An extremely similar computation is carried out in~\cite[Proposition~2.13]{FiM} as well as in other sources; the key observations are that the $X_\gamma$ are independent (so that the characteristic function of the double sum defining $X_i$ is the product of the individual characteristic functions) and that the characteristic function of~$cX_\gamma$ equals~$J_0(cz)$ for any constant~$c\in\R$.
\end{proof}

We proceed as in~\cite[Propositions~2.10--2.12]{FiM}, writing the power series of the logarithms of these infinite products in terms of the following quantities.

\begin{definition} \label{Wi def}
For $i\in\{1,2,3\}$ and any positive integer~$m$, define
\[
W_i(m) = \frac {8^{2m}|\lambda_{2m}|} {V_i}  \sum_{\chi \in H_i} \sum_{\substack{  \gamma > 0 \\ L(\frac 1 2 + i\gamma, \chi) = 0}}  \frac 1 {(\frac14+\gamma^2)^m},
\]
with $\lambda_n$ as in Definition~\ref{Bessel def}.
\end{definition} 

\begin{lemma} \label{W bound lemma}
We have $W_i(1) = \frac12$ and $W_i(m) \ll (\frac{20}3)^{2m}$ for all $m\ge2$.
\end{lemma}

\begin{proof}
We compute directly from Definition~\ref{Bessel def} that $\lambda_2=-\frac14$, and thus
\begin{align*}
W_i(1) &= \frac {8^{2}|\lambda_{2}|} {V_i} \cdot \sum_{\chi \in H_i} \sum_{\substack{  \gamma > 0 \\ L(\frac 1 2 + i\gamma, \chi) = 0}} \frac1{\frac14+\gamma^2} = \frac{16}{V_i} \cdot \frac{V_i}{32} = \frac12
\end{align*}
by Definition~\ref{V def}. On the other hand, since $\frac14+\gamma^2\ge\frac14$, by Lemma~\ref{lambda lemma} we have
\begin{align*}
W_i(m) &= \frac {8^{2m}|\lambda_{2m}|} {V_i}  \sum_{\chi \in H_i} \sum_{\substack{  \gamma > 0 \\ L(\frac 1 2 + i\gamma, \chi) = 0}}  \frac 1 {(\frac14+\gamma^2)^m} \\
&\ll \frac {8^{2m}} {V_i} \bigg( \frac 5 {12} \bigg)^{2m}  \sum_{\chi \in H_i} \sum_{\substack{  \gamma > 0 \\ L(\frac 1 2 + i\gamma, \chi) = 0}}  \frac {4^{m-1}} {\frac14+\gamma^2} \\
&= \frac1{V_i} \bigg( \frac{10}3 \bigg)^{2m} 4^{m-1} \cdot \frac{V_i}{32} = \frac1{128} \bigg(\frac {20} 3 \bigg)^{2m}. \qedhere
\end{align*}
\end{proof}

The next lemma codifies the standard fact that power series can be estimated by their first terms in compact subsets inside their open disks of convergence.

\begin{lemma} \label{general series lemma}
Let $i\in\{1,2,3\}$. For any integers $k\ge0$ and $d\ge-2k$ and any polynomial $P(x)$,
\begin{align*}
\sum_{m=k}^\infty P(m) W_i(m)z^{2m+d} &\ll_{P,k} |z|^{2k+d} \quad\text{uniformly for $|z| \le \tfrac1{10}$}, \\
\exp\bigg( {-}V_i \sum_{m=k}^\infty P(m) W_i(m)z^{2m+d} \bigg) &= 1+O_{P,k}(V_i |z|^{2k+d}) \quad\text{uniformly for $|z| \le V_i^{-1/(2k+d)}$}.
\end{align*}
\end{lemma}

\begin{proof}
By Lemma~\ref{W bound lemma}, we know that $P(m) W_i(m) \ll |P(m)| (\frac{20}3)^{2m} \ll_P 7^m$. Thus when $|z|\le\frac1{10}$, we obtain
\[
\sum_{m=k}^\infty P(m) W_i(m)z^{2m+d} \ll_P \sum_{m=k}^\infty 7^m \bigg( \frac1{10} \bigg)^{2(m-k)} |z|^{2k+d} = \frac{100}{93} 7^k |z|^{2k+d},
\]
which suffices for the first bound. The second bound follows from the first because $e^w = 1+O_{P,k}(|w|)$ uniformly for $|w| \ll_{P,k} 1$.
\end{proof}

The final result of this section is the connection between the characteristic functions $\Phi_i(z)$ and the quantities $W_i(m)$.

\begin{proposition} \label{Phi exp series prop}
For $|z| < \frac3{20}$ and $i\in\{1,2,3\}$, we have $\Phi_i(z) = \exp \big( {-}V_i\sum_{m=1}^\infty W_i(m)z^{2m} \big)$.
\end{proposition}

\begin{proof}
From Definition~\ref{F Phi def},
\begin{align*}
\log \Phi_i(z) &= \log \bigg( \prod_{\chi \in H_i} F(4z,\chi) \bigg) = \sum_{\chi \in H_i} \sum_{\substack{\gamma>0 \\ L(\frac12+i\gamma,\chi)=0}} \log J_0\bigg( \frac{8z}{\sqrt{\frac14+\gamma^2}} \bigg).
\end{align*}
For $|z| < \frac3{20}$, the argument of $J_0$ is less than $8 \cdot \frac 3 {20} / {\frac 1 2} = \frac {12} 5$, and so the power series expansion of $\log J_0$ converges absolutely by Lemma~\ref{lambda lemma}(a), giving
\[
\log \Phi_i(z) =  \sum_{\chi \in H_i} \sum_{\substack{\gamma>0 \\ L(\frac12+i\gamma,\chi)=0}} \sum_{n=0}^\infty \lambda_n \bigg( \frac{8z}{\sqrt{\frac14+\gamma^2}} \bigg)^n.
\]
By Lemma~\ref{lambda lemma}(b)--(c), the $n=0$ term and the terms with~$n$ odd vanish, and we may change $\lambda_n$ for~$n$ even to $-|\lambda_{2m}|$. Since the sum over~$\gamma$ converges absolutely for $n\ge2$, we may rearrange terms to obtain
\begin{align*}
\log \Phi_i(z) &= - \sum_{m=1}^\infty |\lambda_{2m}| (8z)^{2m} \sum_{\chi \in H_i}  \sum_{\substack{\gamma>0 \\ L(\frac12+i\gamma,\chi)=0}} \frac 1 {(\frac14+\gamma^2)^m} \\
&= - V_i \sum_{m=1}^\infty z^{2m} |\lambda_{2m}| \frac{8^{2m}}{V_i}  \sum_{\chi \in H_i} \sum_{\substack{\gamma>0 \\ L(\frac12+i\gamma,\chi)=0}} \frac 1 {(\frac14+\gamma^2)^m} = - V_i \sum_{m=1}^\infty z^{2m} W_i(m)
\end{align*}
by Definition~\ref{Wi def}.
\end{proof}

\section{Comparison of characteristic functions} \label{zeroth deriv section}

The goal of this section is to obtain pointwise bounds for the difference between the characteristic function $\Phi_i(x)$ and the characteristic function $e^{-V_ix^2/2}$ of a normal random variable with mean~$0$ and variance~$V_i$. We begin by establishing the asymptotic sizes of these variances.

\begin{lemma} \label{bchi lemma}
Assume GRH. If $q \ge 3$, then $b(\chi) = \log q^* + O(\log\log q)$, where $q^*$ is the conductor of~$\chi$.
\end{lemma}

\begin{proof}
According to~\cite[Lemma~3.5 and the proof of Proposition~3.6]{FiM}, on GRH we have
\begin{align*}
b(\chi) &= \log{q^*}+2\Re \frac {L'(1,\chi^*)} {L(1,\chi^*)}+O(1) \\
&= \log{q^*}+O(\log\log{q^*})+O(1) = \log{q^*}+O(\log\log q)
\end{align*}
for $q^* >1$. Since $b(\chi) = O(1)$ when $q^* = 1$, we conclude that $b(\chi) = \log q^* + O(\log\log q)$ for all characters $\chi\mod q$.
\end{proof}

We can now establish the sizes of the quantities from Definition~\ref{V def}.

\begin{proposition} \label{V and eta sizes prop}
Assume GRH. We have $V(q) = \phi(q) \log q + O(\phi(q)\log\log q)$ and $V_i = 4\phi(q) \log q + O(\phi(q)\log\log q)$. In particular, $\eta_i \ll (\log\log q)/{\log q}$.
\end{proposition}

\begin{proof}
It suffices to prove the two asymptotic formulas, as then the estimate for $\eta_i$ follows directly from Definition~\ref{V def}.

If $S$ is any set of characters\mod q such that $\chi\in S$ if and only if $\overline\chi\in S$, then
\[
2\sum_{\chi\in S} b_+(\chi) = \sum_{\chi\in S} b_+(\chi) + \sum_{\overline\chi\in S} b_+(\overline\chi) = \sum_{\chi\in S} b_+(\chi) + \sum_{\chi\in S} b_+(\overline\chi) = \sum_{\chi\in S} b(\chi)
\]
as noted in Definition~\ref{b+ def}. Then, by Lemma~\ref{bchi lemma},
\begin{align}
2\sum_{\chi\in S} b_+(\chi) &= \sum_{\chi\in S} \big( \log q^* + O(\log \log q) \big) \notag \\
&= \sum_{\chi\in S} \log q - \sum_{\chi\in S} ( \log q - \log q^* ) + O(\#S \log \log q) \notag \\
&= \#S \log q + O\bigg( \sum_{\chi\in S} ( \log q - \log q^* ) \bigg) + O(\phi(q) \log \log q). \label{general b+ sum}
\end{align}
However,~\cite[Proposition~3.3 and the proof of Proposition~3.6]{FiM} imply that
\[
\sum_{\chi\mod q} ( \log q - \log q^* ) = \phi(q) \sum_{p\mid q} \frac{\log p}{p-1} \ll \phi(q)\log \log q,
\]
and so (since $\log q - \log q^*$ is nonnegative) the first error term can be absorbed into the second.

In particular, combining Definition~\ref{V def} with equation~\eqref{general b+ sum} yields
\begin{align*}
V(q) &= 2\sum_{\substack{\chi\mod q \\ \chi\ne\chi_0}} b_+(\chi) = (\phi(q)-1) \log q + O(\phi(q) \log \log q) \\
V_i &= 32\sum_{\chi\in H_i} b_+(\chi) = 16 \big( \tfrac14\phi(q) \log q + O(\phi(q) \log \log q) \big),
\end{align*}
since $\#H_i = \frac14\phi(q)$ by Remark~\ref{Hi remark}.
\end{proof}

In our proofs we will need~$q$ to be sufficiently large for some of our inequalities to hold; the following quantity~$q_0$ will be used through the end of Section~\ref{finally done section}. (Our justification that~$q_0$ exists references a result that assumed GRH, but the existence of~$q_0$ could be justified without that hypothesis; correspondingly, we do not include ``assume~GRH'' in the results of this section when the only detail for which GRH is required is the appearance of~$q_0$.)

\begin{definition} \label{q0 def}
We define a positive real number~$q_0$ as follows. By Proposition~\ref{V and eta sizes prop} we know that $V_i \gg \phi(q)\log q$ uniformly for all choices of $a_i,a_j,a_k$ from Definition~\ref{a def}. Therefore, we can choose $q_0>0$ so that $V_i \ge \max\{2^{20},\phi(q)\}$ for all $q>q_0$. We will often use (without comment) the specific consequence that $V_i^{-1/4} \le \frac1{32}$ for $q>q_0$.
\end{definition}

We proceed now to establish several estimates for $\Phi_i(x)$ valid for various ranges of~$x$. The first such formula, for arguments close to~$0$, is similar to~\cite[Proposition 2.12]{FiM}.

\begin{proposition} \label{small range prop}
For $i\in\{1,2,3\}$ and $q>q_0$, we have $\Phi_i(z) = e^{-V_iz^2/2} \big( 1+O( V_i|z|^4 ) \big)$ for all complex numbers~$z$ with $|z|\le V_i^{-1/4}$.
\end{proposition}

\begin{proof}
Since
\[
\Phi_i(z) = \exp\bigg( {-}V_i \sum_{m=1}^\infty W_i(m)z^{2m} \bigg) = e^{-V_i W_i(1)z^2} \exp\bigg( {-}V_i \sum_{m=2}^\infty W_i(m)z^{2m} \bigg)
\]
by Proposition~\ref{Phi exp series prop}, and $W_i(1)=\frac12$ by Lemma~\ref{W bound lemma}, the estimate $\Phi_i(z) = e^{-V_iz^2/2} \exp( O( V_i|z|^4 ))$ (which implies the asserted statement) follows immediately from Lemma~\ref{general series lemma}.
\end{proof}

\begin{lemma} \label{Phi sorta decreasing}
For $i\in\{1,2,3\}$ and $q>q_0$, we have $|\Phi_i(x)| \ll e^{-V_i^{1/2}/2}$ for all real numbers~$x$ with $|x| \ge V_i^{-1/4}$. 
\end{lemma}

\begin{proof}
From Definition~\ref{F Phi def},
\[
\Phi_i(x) = \prod_{\chi \in H_i} \prod_{\substack{\gamma>0 \\ L(\frac12+i\gamma,\chi)=0}} J_0\bigg( \frac{8x}{\sqrt{\frac14+\gamma^2}} \bigg).
\]
In each factor, set $\kappa = 8V_i^{-1/4}/\sqrt{\frac14+\gamma^2}$. We have $0 \le \kappa \le 8V_i^{-1/4}/\frac12 \le 16\cdot\frac1{32} = \frac12$ since $q>q_0$, and therefore Lemma~\ref{Bessel sorta decreasing lemma} applies to each factor, yielding
\[
|\Phi_i(x)| = \prod_{\chi \in H_i} \prod_{\substack{\gamma>0 \\ L(\frac12+i\gamma,\chi)=0}} \bigg| J_0\bigg( \frac{8x}{\sqrt{\frac14+\gamma^2}} \bigg) \bigg| \le \prod_{\chi \in H_i} \prod_{\substack{\gamma>0 \\ L(\frac12+i\gamma,\chi)=0}} J_0\bigg( \frac{8V_i^{-1/4}}{\sqrt{\frac14+\gamma^2}} \bigg) = \Phi_i(V_i^{-1/4})
\]
for $|x| \ge V_i^{-1/4}$. The lemma now follows from the estimate $\Phi_i(V_i^{-1/4}) \ll e^{-V_i^{1/2}/2}$ which is a special case of Proposition~\ref{small range prop}.
\end{proof}

The following proposition could be proved directly (derived from~\cite[Lemma~2.16]{FiM}, for example); however, we will need a more general result later, so it is more efficient to derive this proposition from that later result.

\begin{lemma} \label{Phi bound large x}
For $i\in\{1,2,3\}$, we have $|\Phi_i(x)| < e^{-\phi(q)|x|/8}$ for all real numbers~$x$ with $|x| \ge 50$. 
\end{lemma}

\begin{proof}
The lemma follows immediately from equation~\eqref{Phi as Ui product} and Lemma~\ref{Bessel product missing a few} with $A=\emptyset$.
\end{proof}

Ultimately we will want to compare the characteristic function $\Phi_i(x)$ to the characteristic function $e^{-V_ix^2/2}$ of a normal random variable with mean~$0$ and variance~$V_i$. The following proposition summaries the results of this section in that light.

\begin{proposition} \label{three ranges prop}
For $i\in\{1,2,3\}$ and $q>q_0$, and for any real number~$x$,
\[
\Phi_i(x) - e^{-V_ix^2/2} \ll \begin{cases}
V_ix^4 e^{-V_ix^2/2}, & \text{if } |x| \le V_i^{-1/4}, \\
e^{-V_i^{1/2}/2}, & \text{if } V_i^{-1/4} \le |x| \le 50, \\
e^{-\phi(q)|x|/8}, & \text{if } |x| \ge 50.
\end{cases}
\]
\end{proposition}

\begin{proof}
The first assertion is immediate from Proposition~\ref{small range prop}. For the second and third assertions, we use $|\Phi_i(x) - e^{-V_ix^2/2}| \le |\Phi_i(x)| + e^{-V_ix^2/2}$, and note that the term $e^{-V_ix^2/2}$ is insignificant compared to the asserted estimates (due to the range of~$x$ in the second case and the definition of~$q_0$ in the third case). Therefore the second assertion follows from Lemma~\ref{Phi sorta decreasing} while the third assertion follows from Lemma~\ref{Phi bound large x}.
\end{proof}

\section{Comparison of second derivatives of characteristic functions} \label{second deriv section}

We continue to use the methods of the previous section, now with the goal of providing analogous bounds for $\Phi_i''(x)$ for various ranges of~$x$, with an eye towards an eventual comparison with the second derivative of $e^{-V_ix^2/2}$. We are fortunate to have access to several different representations of $\Phi_i''(x)$, as no one of them will be entirely sufficient for our needs. We begin with the following power series representation.

\begin{lemma} \label{Phi'' series lemma}
For $i\in\{1,2,3\}$ and $|z| < \frac3{20}$, 
\begin{equation*}
\Phi''(z) = \Phi_i(z) \bigg\{ \bigg( V_i\sum_{m=1}^\infty 2mW_i(m)z^{2m-1} \bigg)^2 - V_i\sum_{m=1}^\infty 2m(2m-1)W_i(m)z^{2m-2} \bigg\}.
\end{equation*}
\end{lemma}

\begin{proof}
We know that $(e^{h(z)})'' = e^{h(z)} \big( h'(z)^2 + h''(z) \big)$ for any smooth function~$h(z)$. The proposition follows from applying this identity with $h(z)$ equal to the power series in the exponent of the formula for $\Phi_i(z)$ given in Proposition~\ref{Phi exp series prop}, which can be differentiated term-by-term on any open set on which it converges.
\end{proof}

\begin{lemma} \label{small range '' lemma}
For $i\in\{1,2,3\}$ and $q>q_0$, we have
\[
\Phi''(z) = e^{-V_iz^2/2} \big( V_i^2 z^2 - V_i + O(V_i|z|^2 + V_i^3 |z|^6) \big)
\]
for $|z|\le V_i^{-1/4}$.
\end{lemma}

\begin{proof}
Using Proposition~\ref{small range prop} followed by Lemma~\ref{general series lemma} twice, we see that for $|z|\le V_i^{-1/4}$,
\begin{align*}
\Phi_i(z) &= e^{-V_iz^2/2}\big( 1+O( V_i|z|^4 ) \big) \\
\sum_{m=1}^\infty 2mW_i(m)z^{2m-1} &= z + O(|z|^3) \\
\sum_{m=1}^\infty 2m(2m-1)W_i(m)z^{2m-2} &= 1 + O(|z|^2)
\end{align*}
since $W_i(1) = \frac12$. (The second and third formulas require $|z| \le \frac1{10}$, which is implied by $|z|\le V_i^{-1/4}$ since $q>q_0$.)
Therefore by Lemma~\ref{Phi'' series lemma}, for $|z|\le V_i^{-1/4}$ we have
\begin{align}
\Phi_i(z) \bigg( V_i\sum_{m=1}^\infty 2mW_i(m)z^{2m-1} \bigg)^2 &= e^{-V_iz^2/2}\big( 1+O( V_i|z|^4 ) \big) \big( V_i(z + O(|z|^3)) \big)^2 \notag \\
&= e^{-V_iz^2/2} \big( V_i^2z^2 + O(V_i^2|z|^4 + V_i^3|z|^6)) \big) \label{for use later}
\end{align}
and ultimately
\begin{align*}
\Phi''(z) &= e^{-V_iz^2/2}\big( 1+O( V_i|z|^4 ) \big) \big\{ \big( V_i(z + O(|z|^3)) \big)^2 - (V_i(1 + O(|z|^2)) \big\} \\
&= e^{-V_iz^2/2}\big( 1+O( V_i|z|^4 ) \big) \big( V_i^2 z^2 - V_i + O(V_i|z|^2 + V_i^2 |z|^4) \big) \\
&= e^{-V_iz^2/2} \big( V_i^2 z^2 - V_i + O(V_i|z|^2 + V_i^2 |z|^4 + V_i^3 |z|^6) \big),
\end{align*}
which implies the statement of the proposition since $V_i^2 |z|^4$ is always dominated by one of the other two error terms.
\end{proof}

In the proof of Lemma~\ref{Phi sorta decreasing}, we used the fact that $\Phi(x)$ was a simple product of terms all of which were positive, and took their largest values, near the origin. The corresponding expression for $\Phi_i''(x)$ is more complicated, however, and involves functions whose values near the origin have both signs. We therefore establish a particular decomposition of $\Phi_i''(x)$ into two pieces, each of which has the unanimity of sign necessary for us to infer from Lemma~\ref{Bessel sorta decreasing lemma} that its largest values are near the origin.

It will be convenient to define the set of ordinates
\begin{equation} \label{Ui def}
U_i = \bigcup_{\chi \in H_i}\{ \gamma>0 \colon  L(1/2+i\gamma, \chi) = 0 \}
\end{equation}
that indexes the infinite product that defines $\Phi_i(z)$.

\begin{lemma} \label{Phi'' decomposition lemma}
For $i\in\{1,2,3\}$ and $z\in\C$ we may write
\begin{equation} \label{eq1}
\Phi''_i(z) = z^2 \Psi_i(z) + \Theta_i(z),
\end{equation}
where
\begin{align}
\Psi_i(z) &= \sum_{\substack{\gamma_1, \gamma_2 \in U_i \\ \gamma_1\ne\gamma_2}} \frac {64} {{\frac14+\gamma_1^2}} K\bigg( \frac {8z} {\sqrt{\frac14+\gamma_1^2}}\bigg)  \frac {64} {{\frac14+\gamma_2^2}} K\bigg( \frac {8z} {\sqrt{\frac14+\gamma_2^2}}\bigg) \prod_{\gamma \in U_i \setminus \{\gamma_1,\gamma_2\}} J_0\bigg( \frac {8z} {\sqrt{\frac14+\gamma^2}}\bigg) \label{Psii def} \\
\Theta_i(z) &= \sum_{\gamma_1 \in U_i}  \frac {64} {\frac14+\gamma_1^2} J_0''\bigg( \frac {8z} {\sqrt{\frac14+\gamma_1^2}}\bigg) \prod_{\gamma \in U_i \setminus \{\gamma_1\}} J_0\bigg( \frac {8z} {\sqrt{\frac14+\gamma^2}}\bigg). \label{Thetai def}
\end{align}
\end{lemma}

\begin{proof}
By Definition~\ref{F Phi def},
\begin{equation} \label{Phi as Ui product}
\Phi_i(z) = \prod_{\chi \in H_i} \prod_{\substack {\gamma > 0 \\ L(1/2+i\gamma, \chi) = 0}}  J_0 \bigg( \frac {8z} {\sqrt{\frac14+\gamma^2}}\bigg) = \prod_{\gamma \in U_i} J_0 \bigg( \frac {8z} {\sqrt{\frac14+\gamma^2}}\bigg).
\end{equation}
This infinite product of analytic functions converges uniformly in any bounded subset of~$\C$ (since $J_0(z) = 1+O(z^2)$ near $z=0$ and the series $\sum_{\gamma\in U_i} 1/\gamma^2$ converges).
Therefore we may differentiate the infinite product by applying the product rule twice:
\begin{align*}
\Phi''(z) &= \sum_{\gamma_1 \in U_i}  \frac {d^2} {dz^2} J_0\bigg( \frac {8z} {\sqrt{\frac14+\gamma_1^2}}\bigg) \prod_{\gamma \in U_i \setminus \{\gamma_1\}} J_0\bigg( \frac {8z} {\sqrt{\frac14+\gamma^2}}\bigg)\\
&\qquad{}+  \sum_{\substack{\gamma_1, \gamma_2 \in U_i \\ \gamma_1\ne\gamma_2}} \frac d {dz} J_0\bigg( \frac {8z} {\sqrt{\frac14+\gamma_1^2}}\bigg)  \frac d {dz} J_0\bigg( \frac {8z} {\sqrt{\frac14+\gamma_2^2}}\bigg) \prod_{\gamma \in U_i \setminus \{\gamma_1,\gamma_2\}} J_0\bigg( \frac {8z} {\sqrt{\frac14+\gamma^2}}\bigg) \\
&= \sum_{\gamma_1 \in U_i}  \frac {64} {\frac14+\gamma_1^2} J_0''\bigg( \frac {8z} {\sqrt{\frac14+\gamma_1^2}}\bigg) \prod_{\gamma \in U_i \setminus \{\gamma_1\}} J_0\bigg( \frac {8z} {\sqrt{\frac14+\gamma^2}}\bigg) \\
&\qquad{}+  \sum_{\substack{\gamma_1, \gamma_2 \in U_i \\ \gamma_1\ne\gamma_2}} \frac {8} {\sqrt{\frac14+\gamma_1^2}} J_0'\bigg( \frac {8z} {\sqrt{\frac14+\gamma_1^2}}\bigg)  \frac {8} {\sqrt{\frac14+\gamma_2^2}} J_0'\bigg( \frac {8z} {\sqrt{\frac14+\gamma_2^2}}\bigg) \prod_{\gamma \in U_i \setminus \{\gamma_1,\gamma_2\}} J_0\bigg( \frac {8z} {\sqrt{\frac14+\gamma^2}}\bigg).
\end{align*}
Consulting Definition~\ref{K def} reveals that this last expression is the same as equation~\eqref{eq1} (the negative signs in the definition of $K(t)$ come in pairs).
\end{proof}

To efficiently bound, for small~$|z|$, the first component $z^2 \Psi_i(z)$ in the above decomposition of~$\Phi''(z)$, we need to first write it in a different form. Recall the function $D(x)$ from Definition~\ref{K def}.

\begin{lemma} \label{x^2 Psi different form lemma}
For $i\in\{1,2,3\}$ and complex numbers~$z$ satisfying $|z|\le\frac3{20}$,
\begin{equation} \label{x^2 Psi different form}
z^2 \Psi_i(z) = \Phi_i(z) \bigg\{ \bigg( \sum_{\gamma\in U_i} \frac{8}{\sqrt{\frac14+\gamma^2}} D\bigg( \frac{8z}{\sqrt{\frac14+\gamma^2}} \bigg) \bigg)^2 - \sum_{\gamma\in U_i} \frac{64}{\frac14+\gamma^2} D\bigg( \frac{8z}{\sqrt{\frac14+\gamma^2}} \bigg)^2 \bigg\}.
\end{equation}
In particular, for real numbers $x$ satisfying $|x|\le\frac1{32}$,
\begin{equation} \label{x^2 Psi upper bound}
x^2 \Psi_i(x) \le \Phi_i(x) \bigg\{ \bigg( \sum_{\gamma\in U_i} \frac{8}{\sqrt{\frac14+\gamma^2}} D\bigg( \frac{8x}{\sqrt{\frac14+\gamma^2}} \bigg) \bigg)^2.
\end{equation}
\end{lemma}

\begin{proof}
Again we use $(e^{h(z)})'' = e^{h(z)} \big( h'(z)^2 + h''(z) \big)$, this time with $h(z)$ equal to the infinite series in the identity
\begin{align*}
\Phi_i(z) &= \exp\bigg( \sum_{\gamma\in U_i} \log J_0\bigg( \frac{8z}{\sqrt{\frac14+\gamma^2}} \bigg) \bigg),
\end{align*}
valid for $|z|\le\frac3{20}$ since the argument of~$J_0$ does not vanish there;
we obtain
\begin{align}
\Phi''_i(z) &= \Phi_i(z) \bigg( \sum_{\gamma\in U_i} \frac{8}{\sqrt{\frac14+\gamma^2}} \frac{J_0'\big( 8z/\sqrt{\frac14+\gamma^2} \big)}{J_0\big( 8z/\sqrt{\frac14+\gamma^2} \big)} \bigg)^2 \notag \\
&\qquad{}+ \Phi_i(z) \sum_{\gamma\in U_i} \frac{64}{\frac14+\gamma^2} \frac{J_0\big( 8z/\sqrt{\frac14+\gamma^2} \big) J_0''\big( 8z/\sqrt{\frac14+\gamma^2} \big) - J_0'\big( 8z/\sqrt{\frac14+\gamma^2} \big)^2}{J_0\big( 8z/\sqrt{\frac14+\gamma^2} \big)^2} \notag \\
&= \Phi_i(z) \bigg\{ \bigg( \sum_{\gamma\in U_i} \frac{8}{\sqrt{\frac14+\gamma^2}} D\bigg( \frac{8z}{\sqrt{\frac14+\gamma^2}} \bigg) \bigg)^2 - \sum_{\gamma\in U_i} \frac{64}{\frac14+\gamma^2} D\bigg( \frac{8z}{\sqrt{\frac14+\gamma^2}} \bigg)^2 \bigg\} \label{secretly Psi} \\
&\qquad{}+ \Phi_i(z) \sum_{\gamma\in U_i} \frac{64}{\frac14+\gamma^2} \bigg( \frac{J_0''\big( 8z/\sqrt{\frac14+\gamma^2} \big)}{J_0\big( 8z/\sqrt{\frac14+\gamma^2} \big)}\bigg) \notag
\end{align}
by Definition~\ref{K def} (the negative signs all occur in pairs).
However, using equation~\eqref{Phi as Ui product} yields
\begin{align*}
\Phi_i(z) & \sum_{\gamma\in U_i} \frac{64}{\frac14+\gamma^2} \bigg( \frac{J_0''\big( 8z/\sqrt{\frac14+\gamma^2} \big)}{J_0\big( 8z/\sqrt{\frac14+\gamma^2} \big)}\bigg) \\
&= \prod_{\gamma \in U_i} J_0 \bigg( \frac {8z} {\sqrt{\frac14+\gamma^2}}\bigg) \sum_{\gamma_1\in U_i} \frac{64}{\frac14+\gamma_1^2} J_0\bigg(\frac{8z}{\sqrt{\frac14+\gamma_1^2}}\bigg)^{-1} J_0''\bigg(\frac {8z}{\sqrt{\frac14+\gamma_1^2}}\bigg) = \Theta_i(z)
\end{align*}
by equation~\eqref{Thetai def}; thus by the identity~\eqref{eq1}, we conclude that the expression on line~\eqref{secretly Psi} must equal~$z^2 \Psi_i(z)$, establishing the first assertion of the lemma.

As for the second assertion, when~$z=x$ is a real number satisfying $|x|\le\frac1{32}$, all of the summands in the two series in equation~\eqref{secretly Psi} are positive by Lemma~\ref{Bessel sorta decreasing lemma}, since the argument of~$D$ is at most $16|x| \le \frac12$ in absolute value. Notice that the second sum in equation~\eqref{secretly Psi} consists precisely of the squares of the summands from the first sum; in particular, both sums are positive and the second sum is no larger than the square of the first sum. We may therefore ignore the second sum when finding an upper bound, which establishes the second assertion of the lemma.
\end{proof}

\begin{lemma} \label{middle range Psi lemma}
For $i\in\{1,2,3\}$ and $q>q_0$, we have $\Psi_i(x) \ll e^{-V_i^{1/2}/2} V_i^{3/2}$ for $|x| \ge V_i^{-1/4}$.
\end{lemma}

\begin{proof}
In each factor in equation~\eqref{Psii def}, set $\kappa = 8V_i^{-1/4}/\sqrt{\frac14+\gamma^2}$. We have $0 \le \kappa \le 8V_i^{-1/4}/\frac12 \le 16\cdot\frac1{32} = \frac12$ since $q>q_0$, and therefore Lemma~\ref{Bessel sorta decreasing lemma} applies to each factor, yielding
\begin{align*}
|\Psi_i(x)| &= \sum_{\substack{\gamma_1, \gamma_2 \in U_i \\ \gamma_1\ne\gamma_2}} \frac {64} {{\frac14+\gamma_1^2}} K\bigg( \frac {8x} {\sqrt{\frac14+\gamma_1^2}}\bigg)  \frac {64} {{\frac14+\gamma_2^2}} K\bigg( \frac {8x} {\sqrt{\frac14+\gamma_2^2}}\bigg) \prod_{\gamma \in U_i \setminus \{\gamma_1,\gamma_2\}} J_0\bigg( \frac {8x} {\sqrt{\frac14+\gamma^2}}\bigg) \\
&\le \sum_{\substack{\gamma_1, \gamma_2 \in U_i \\ \gamma_1\ne\gamma_2}} \frac {64} {{\frac14+\gamma_1^2}} K\bigg( \frac {8V_i^{-1/4}} {\sqrt{\frac14+\gamma_1^2}}\bigg)  \frac {64} {{\frac14+\gamma_2^2}} K\bigg( \frac {8V_i^{-1/4}} {\sqrt{\frac14+\gamma_2^2}}\bigg) \prod_{\gamma \in U_i \setminus \{\gamma_1,\gamma_2\}} J_0\bigg( \frac {8V_i^{-1/4}} {\sqrt{\frac14+\gamma^2}}\bigg) \\
&= \Psi_i(V_i^{-1/4})
\end{align*}
for $|x| \ge V_i^{-1/4}$.
Since $V_i^{-1/4} \le \frac1{32}$ when $q>q_0$, we may apply the upper bound~\eqref{x^2 Psi upper bound} to obtain
\[
\Psi_i(V_i^{-1/4}) \le \Phi_i(V_i^{-1/4}) \bigg( \sum_{\gamma\in U_i} \frac{8}{\sqrt{\frac14+\gamma^2}} D\bigg( \frac{8V_i^{-1/4}}{\sqrt{\frac14+\gamma^2}} \bigg) \bigg)^2.
\]
It follows from Definitions~\ref{Bessel def} and~\ref{K def} that
\begin{align*}
D(t) = -\frac{J_0'(t)}{J_0(t)} = -\frac d{dt} \log J_0(t) = \sum_{m=1}^\infty |\lambda_{2m}| 2m t^{2m-1}
\end{align*}
for $|t| \le \frac{12}5$, and so
\begin{align*}
\Psi_i(V_i^{-1/4}) &\le  \Phi_i(V_i^{-1/4}) \bigg( \sum_{\gamma\in U_i} \frac{8}{\sqrt{\frac14+\gamma^2}} \sum_{m=1}^\infty |\lambda_{2m}| 2m \bigg( \frac{8V_i^{-1/4}}{\sqrt{\frac14+\gamma^2}} \bigg)^{2m-1} \bigg)^2 \\
&= \Phi_i(V_i^{-1/4}) \bigg( \sum_{m=1}^\infty |\lambda_{2m}| 2m \cdot 8^{2m}V_i^{-(2m-1)/4} \sum_{\gamma\in U_i} \frac{1}{(\frac14+\gamma^2)^m} \bigg)^2 \\
&= \Phi_i(V_i^{-1/4}) \bigg( V_i \sum_{m=1}^\infty 2mW_i(m)V_i^{-(2m-1)/4} \bigg)^2 \ll \Phi_i(V_i^{-1/4}) \big( V_i \cdot V_i^{-1/4} \big)^2
\end{align*}
by Lemma~\ref{general series lemma}. The statement of the proposition now follows from Proposition~\ref{small range prop}.
\end{proof}

The estimates we have derived for $\Phi_i''(x)$ and $\Psi_i(x)$ for small~$|x|$ imply a similar estimate for $\Theta_i(x)$ for small~$|x|$; thanks to Lemma~\ref{Bessel sorta decreasing lemma}, we can deduce an estimate for $\Theta_i(x)$ for large~$|x|$, which we can subsequently use to estimate $\Phi_i''(x)$ itself for larger~$|x|$.

\begin{lemma} \label{middle range Theta lemma}
For $i\in\{1,2,3\}$ and $q>q_0$, we have $\Theta_i(x) \ll e^{-V_i^{1/2}/2} V_i^{3/2}$ for $|x| \ge V_i^{-1/4}$.
\end{lemma}

\begin{proof}
In each factor in equation~\eqref{Thetai def}, set $\kappa = 8V_i^{-1/4}/\sqrt{\frac14+\gamma^2}$. We have $0 \le \kappa \le 8V_i^{-1/4}/\frac12 \le 16\cdot\frac1{32} = \frac12$ since $q>q_0$, and therefore Lemma~\ref{Bessel sorta decreasing lemma} applies to each factor, yielding
\begin{align*}
|\Theta_i(x)| &= \sum_{\gamma_1 \in U_i}  \frac {64} {\frac14+\gamma_1^2} J_0''\bigg( \frac {8x} {\sqrt{\frac14+\gamma_1^2}}\bigg) \prod_{\gamma \in U_i \setminus \{\gamma_1\}} J_0\bigg( \frac {8x} {\sqrt{\frac14+\gamma^2}}\bigg) \\
&\le \sum_{\gamma_1 \in U_i}  \frac {64} {\frac14+\gamma_1^2} J_0''\bigg( \frac {8x} {\sqrt{\frac14+\gamma_1^2}}\bigg) \prod_{\gamma \in U_i \setminus \{\gamma_1\}} J_0\bigg( \frac {8x} {\sqrt{\frac14+\gamma^2}}\bigg) = \Theta_i(V_i^{-1/4})
\end{align*}
for $|x| \ge V_i^{-1/4}$. On the other hand, by the identity~\eqref{eq1} and Lemmas~\ref{small range '' lemma} and~\ref{middle range Psi lemma} we have
\begin{align*}
\Theta_i(V_i^{-1/4}) &= \Phi_i''(V_i^{-1/4}) - (V_i^{-1/4})^2\Psi_i(V_i^{-1/4}) \\
&\ll e^{-V_i^{1/2}/2} V_i^{3/2} + V_i^{-1/2} e^{-V_i^{1/2}/2} V_i^{3/2} \ll e^{-V_i^{1/2}/2} V_i^{3/2}
\end{align*}
as desired.
\end{proof}

\begin{lemma} \label{middle range Phi '' lemma}
For $i\in\{1,2,3\}$ and $q>q_0$, we have $\Phi_i''(x) \ll e^{-V_i^{1/2}/2} V_i^{3/2}$ for $V_i^{-1/4} \le |x| \le 50$.
\end{lemma}

\begin{proof}
The lemma follows immediately from the identity~\eqref{eq1} and Lemmas~\ref{middle range Psi lemma} and~\ref{middle range Theta lemma}, since $x \ll 1$ by assumption.
\end{proof}

Lastly, we use a standard method to estimate $\Phi_i''(x)$ for the largest values of~$|x|$. The proof is complicated only slightly by the fact that the relevant infinite products of Bessel functions are missing a small number of terms after the differentiations; the following lemma provides a serviceable bound that uniformly takes such omitted terms into account.

\begin{lemma} \label{Bessel product missing a few}
Fix $i\in\{1,2,3\}$. If $A$ is any finite subset of $U_i$, then for $|x| \ge 50$,
\[
\prod_{\gamma \in U_i \setminus A} J_0\bigg( \frac {8x} {\sqrt{\frac14+\gamma^2}}\bigg) < 2^{\#A} e^{-\phi(q)|x|/8}.
\]
\end{lemma}

\begin{proof}
Since both sides are even functions of~$x$, we may assume that~$x\ge50$.
Let $N(T,\chi)$ denote the number of nontrivial zeros of $L(s,\chi)$ having imaginary part between~$-T$ and~$T$. 
By \cite[Proposition~2.5]{EBPAP}, for $T\ge150$,
\begin{equation} \label{eq5}
N(T,\chi) \ge \bigg( \frac T\pi - 0.399\bigg) \log {\frac {q^*T}{2\pi e}} - 5.338 \ge \bigg( \frac T\pi - 0.399\bigg) \log {\frac T{2\pi e}} - 5.338 > \frac T2
\end{equation}
where $q^*\ge1$ is the conductor of~$\chi$.
From the classical inequality (see~\cite[Theorem 7.31.2]{Sz})
\[
|J_0(x)| \le \min \bigg\{1,\sqrt{\frac 2 {\pi |x|}} \bigg\},
\]
we see that
\[
\bigg| \prod_{\gamma \in U_i \setminus A} J_0\bigg( \frac {8x} {\sqrt{\frac14+\gamma^2}}\bigg)\bigg| \le \prod_{\substack{\gamma < 3x  \\ \gamma \in U_i \setminus A }} \bigg| J_0\bigg( \frac {8x} {\sqrt{\frac14+\gamma^2}}\bigg)\bigg| \le \prod_{\substack{\gamma < 3x  \\ \gamma \in U_i \setminus A }} \frac {(\frac14+\gamma^2)^{1/4}}{2\sqrt{\pi x}}.
\]
One can easily check that when $|x| \ge 50$ and $|\gamma| < 3x$, the factor $(\frac14+\gamma^2)^{1/4}/2\sqrt{\pi x}$ is always less than~$\frac12$.

If we define $N_+(T,\chi)$ to be the number of nontrivial zeros of $L(s,\chi)$ having imaginary part between~$0$ and~$T$, then $N(T,\chi) = N_+(T,\chi) + N_+(T,\overline\chi)$ by the functional equation. Since $\chi\in H_i$ if and only if $\overline\chi\in H_i$, the number of factors in the product is
\begin{align*}
\sum_{\chi\in H_i} & N_+(3x,\chi) - \#\{A\cup(-3x,3x)\} \\
&= \frac12 \bigg( \sum_{\chi\in H_i} N_+(3x,\chi) + \sum_{\overline\chi\in H_i} N_+(3x,\overline\chi) \bigg) - \#\{A\cup(-3x,3x)\} \\
&= \frac12 \sum_{\chi\in H_i} N(3x,\chi) - \#\{A\cup(-3x,3x)\} \ge \frac{\phi(q)}8 \frac{3x}2 - \#A.
\end{align*}
So
\[
\prod_{\substack{\gamma < 3x  \\ \gamma \in U_i \setminus A }} \frac {(\frac14+\gamma^2)^{1/4}}{2\sqrt{\pi |x|}} \le 2^{-(3\phi(q)x/16-\#A)} < 2^{\#A} e^{-\phi(q)x/8},
\]
since $\#H_i = \phi(q)/4$ by Remark~\ref{Hi remark}.
\end{proof}

\begin{lemma} \label{large range '' lemma}
For $i\in\{1,2,3\}$, we have $\Phi_i''(x) \ll V_i^2x^2 e^{-\phi(q)|x|/8}$ for $|x| \ge 50$. 
\end{lemma}

\begin{proof}
From equations~\eqref{Psii def} and~\eqref{Thetai def},
\begin{align*}
\Psi_i(x) &= \sum_{\substack{\gamma_1, \gamma_2 \in U_i \\ \gamma_1\ne\gamma_2}} \frac {64} {{\frac14+\gamma_1^2}} K\bigg( \frac {8x} {\sqrt{\frac14+\gamma_1^2}}\bigg)  \frac {64} {{\frac14+\gamma_2^2}} K\bigg( \frac {8x} {\sqrt{\frac14+\gamma_2^2}}\bigg) \prod_{\gamma \in U_i \setminus \{\gamma_1,\gamma_2\}} J_0\bigg( \frac {8x} {\sqrt{\frac14+\gamma^2}}\bigg), \\
\Theta_i(x) &= \sum_{\gamma_1 \in U_i}  \frac {64} {\frac14+\gamma_1^2} J_0''\bigg( \frac {8x} {\sqrt{\frac14+\gamma_1^2}}\bigg) \prod_{\gamma \in U_i \setminus \{\gamma_1\}} J_0\bigg( \frac {8x} {\sqrt{\frac14+\gamma^2}}\bigg).
\end{align*}
We apply equation~\eqref{Bessel derivative bounds} and Lemma~\ref{Bessel product missing a few} to obtain, for $|x|\ge50$,
\begin{align*}
\Psi_i(x) &\ll \sum_{\substack{\gamma_1, \gamma_2 \in U_i \\ \gamma_1\ne\gamma_2}} \frac {32}{{\frac14+\gamma_1^2}}  \frac{32}{{\frac14+\gamma_2^2}} e^{-\phi(q)|x|/8} \le \bigg( \sum_{\gamma_1 \in U_i}  \frac{32}{\frac14+\gamma_1^2} \bigg)^2 e^{-\phi(q)|x|/8} = V_i^2 e^{-\phi(q)|x|/8} , \\
\Theta_i(x) &\ll \sum_{\gamma_1 \in U_i}  \frac{32}{\frac14+\gamma_1^2} e^{-\phi(q)|x|/8} = V_i e^{-\phi(q)|x|/8}.
\end{align*}
So by equation~\eqref{eq1},
\[
\Phi_i''(x) = x^2 \Psi_i(x) + \Theta_i(x) \ll (V_i^2x^2+V_i) e^{-\phi(q)|x|/8},
\]
which implies the statement of the lemma since $V_i \ll V_i^2x^2$ in this range.
\end{proof}

\begin{remark} \label{Phi' decays}
The method of proof of Proposition~\ref{Phi'' decomposition lemma} gives the expression
\[
\Phi_i'(z) = z \sum_{\gamma_1 \in U_i} \frac {64} {{\frac14+\gamma_1^2}} K\bigg( \frac {8z} {\sqrt{\frac14+\gamma_1^2}}\bigg) \prod_{\gamma \in U_i \setminus \{\gamma_1\}} J_0\bigg( \frac {8z} {\sqrt{\frac14+\gamma^2}}\bigg)
\]
for the first derivative of $\Phi(z)$, from which the estimate $\Phi_i'(x) \ll V_i |x| e^{-\phi(q)|x|/8}$ for $|x| \ge 50$ follows from the method of proof of Lemma~\ref{large range '' lemma}; in particular, $\Phi_i'(x)$ tends to~$0$ as $|x|\to\infty$, a fact we will need later.
\end{remark}

We may now assemble the various bounds for $\Phi_i''(x)$ derived in this section to compare that function to $(V_i^2z^2-V_i) e^{-V_iz^2/2}$, which is the second derivative of the characteristic function $e^{-V_iz^2/2}$ of a normal random variable with mean~$0$ and variance~$V_i$.

\begin{proposition} \label{three ranges '' prop}
For $i\in\{1,2,3\}$ and $q>q_0$, and for any real number $x$,
\[
\Phi_i''(x) - (V_i^2x^2-V_i) e^{-V_ix^2/2} \ll \begin{cases}
V_ix^2 e^{-V_ix^2/2}, & \text{if } |x| \le V_i^{-1/2}, \\
V_i^3x^6 e^{-V_ix^2/2}, & \text{if } V_i^{-1/2} \le |x| \le V_i^{-1/4}, \\
V_i^{3/2} e^{-V_i^{1/2}/2}, & \text{if } V_i^{-1/4} \le |x| \le 50, \\
V_i^2x^2 e^{-\phi(q)|x|/8}, & \text{if } |x| \ge 50.
\end{cases}
\]
\end{proposition}

\begin{proof}
The first two assertions are immediate from Lemma~\ref{small range '' lemma}.
For the third and fourth assertions, we use
\[
|\Phi_i''(x) - (V_i^2x^2-V_i) e^{-V_ix^2/2}| \le |\Phi_i''(x)| + (V_i^2x^2-V_i) e^{-V_ix^2/2} \le |\Phi_i''(x)| + V_i^2x^2 e^{-V_ix^2/2},
\]
and note that the term $V_i^2x^2 e^{-V_ix^2/2}$ is insignificant compared to the asserted estimates (due to the range of~$x$ in the third case and the definition of~$q_0$ in the fourth case). Therefore the third assertion follows from Lemma~\ref{middle range Phi '' lemma} while the fourth assertion follows from Lemma~\ref{large range '' lemma}.
\end{proof}

\section{Comparison of probabilities} \label{final assembly section}

We are now able to estimate the difference between probabilities involving the random variables~$X_i$ from Definition~\ref{Xi def} and normal random variables of the same mean and variance. Using the results of the previous sections, we will bound the integrals of their characteristic functions and the second derivatives thereof over~$\R$; subsequently we will be able to bound the difference between their density functions themselves. We begin with a quick and standard lemma giving the order of magnitude of even moments of a normal distribution.

\begin{lemma} \label{even moments lemma}
For any positive constant~$C$ and any nonnegative integer~$m$,
\begin{align*}
\int_{-\infty}^\infty x^{2m} e^{-Cx^2} \,dx \ll_m \frac1{C^{m+1/2}}.
\end{align*}
\end{lemma}

\begin{proof}
When $m=0$, the formula $\int_{-\infty}^\infty e^{-Cx^2} \,dx = \sqrt{2\pi/C}$ is well known. For $m\ge1$, we integrate by parts to obtain
\begin{align*}
\int_{-\infty}^\infty x^{2m} e^{-Cx^2/2} \,dx &= \int_{-\infty}^\infty x^{2m-1} \cdot xe^{-Cx^2/2} \,dx = \int_{-\infty}^\infty (2m-1)x^{2m-2} \cdot \frac 2C e^{-Cx^2/2} \,dx,
\end{align*}
from which the lemma follows by induction on~$m$.
\end{proof}

\begin{lemma} \label{Phi difference integrals}
Assume GRH. For $i\in\{1,2,3\}$ and $q>q_0$, we have
\begin{align*}
\int_{-\infty}^\infty \big| \Phi_i(x) - e^{-V_ix^2/2} \big| \,dx &\ll V_i^{-3/2} \\
\int_{-\infty}^\infty \big| \Phi_i''(x) - (V_i^2x^2-V_i)e^{-V_ix^2/2} \big| \,dx &\ll V_i^{-1/2}.
\end{align*}
\end{lemma}

\begin{proof}
We write
\begin{multline*}
\int_{-\infty}^\infty \big| \Phi_i(x) - e^{-V_ix^2/2} \big| \,dx = \int_{|x| \le V_i^{-1/4}} \big| \Phi_i(x) - e^{-V_ix^2/2} \big| \,dx \\
+ \int_{V_i^{-1/4} \le |x| \le 50} \big| \Phi_i(x) - e^{-V_ix^2/2} \big| \,dx + \int_{|x| \ge 50} \big| \Phi_i(x) - e^{-V_ix^2/2} \big| \,dx.
\end{multline*}
Using the bounds in Proposition~\ref{three ranges prop},
\begin{align*}
\int_{-\infty}^\infty \big| \Phi_i(x) - e^{-V_ix^2/2} \big| \,dx &\ll \int_{|x| \le V_i^{-1/4}} V_ix^4 e^{-V_ix^2/2} \,dx \\
&\qquad{}+ \int_{V_i^{-1/4} \le |x| \le 50} e^{-V_i^{1/2}/2} \,dx + \int_{|x| \ge 50} e^{-\phi(q)|x|/8} \,dx \\
&\le V_i \int_{-\infty}^\infty x^4 e^{-V_ix^2/2} \,dx + 100e^{-V_i^{1/2}/2} + 2\int_{50}^\infty e^{-\phi(q)x/8} \,dx \\
&\ll V_i \cdot \frac1{V_i^{5/2}} + e^{-V_i^{1/2}/2} + \frac1{\phi(q)} e^{-\phi(q)50/8} \ll V_i^{-3/2}
\end{align*}
by Lemma~\ref{even moments lemma}, where the final simplification uses Proposition~\ref{V and eta sizes prop}.

Similarly, we write
\begin{align*}
\int_{-\infty}^\infty \big| \Phi_i''(x) - (V_i^2x^2-V_i) e^{-V_ix^2/2} \big| \,dx = & \int_{|x| \le V_i^{-1/2}} \big| \Phi_i''(x) - (V_i^2x^2-V_i) e^{-V_ix^2/2} \big| \,dx \\
&+ \int_{V_i^{-1/2} \le |x| \le V_i^{-1/4}} \big| \Phi_i''(x) - (V_i^2x^2-V_i) e^{-V_ix^2/2} \big| \,dx \\
&+ \int_{V_i^{-1/4} \le |x| \le 50} \big| \Phi_i''(x) - (V_i^2x^2-V_i) e^{-V_ix^2/2} \big| \,dx \\
&+ \int_{|x| \ge 50} \big| \Phi_i''(x) - (V_i^2x^2-V_i) e^{-V_ix^2/2} \big| \,dx.
\end{align*}
Using the bounds in Proposition~\ref{three ranges '' prop},
\begin{align*}
\int_{-\infty}^\infty \big| \Phi_i''(x) - e^{-V_ix^2/2} \big| \,dx &\ll \int_{|x| \le V_i^{-1/2}} V_ix^2 e^{-V_ix^2/2} \,dx + \int_{V_i^{-1/2} \le |x| \le V_i^{-1/4}} V_i^3x^6 e^{-V_ix^2/2} \,dx \\
&\qquad{}+ \int_{V_i^{-1/4} \le |x| \le 50} V_i^{3/2} e^{-V_i^{1/2}/2} \,dx + \int_{|x| \ge 50} V_i^2x^2 e^{-\phi(q)|x|/8} \,dx \\
&\le V_i \int_{-\infty}^\infty x^2 e^{-V_ix^2/2} \,dx + V_i^3 \int_{-\infty}^\infty x^6 e^{-V_ix^2/2} \,dx \\
&\qquad{}+ 100V_i^{3/2} e^{-V_i^{1/2}/2} + 2V_i^2 \int_{50}^\infty x^2 e^{-\phi(q)x/8} \,dx \\
&\ll V_i \cdot \frac1{V_i^{3/2}} + V_i^3 \cdot \frac1{V_i^{7/2}} + e^{-V_i^{1/2}/2} + \frac1{\phi(q)} e^{-\phi(q)50/8} \ll V_i^{-1/2}
\end{align*}
again by Lemma~\ref{even moments lemma} and Proposition~\ref{V and eta sizes prop} (and a routine calculation to evaluate the final integral exactly).
\end{proof}

Let $f_i(t)$ denote the density function of the random variable $X_i$ from Definition~\ref{Xi def}, and let $g_i(t) =  (2 \pi V_i)^{-1/2} e^{-t^2/2V_i}$ be the density function of a normal random variable with mean~$0$ and variance~$V_i$. We can bound the difference between these two functions by writing them in terms of their characteristic functions.

\begin{lemma} \label{density pointwise difference}
Assume GRH. For $i\in\{1,2,3\}$ and $q>q_0$, we have
\[
f_i(t) - g_i(t) \ll \min \big\{ V_i^{-3/2}, V_i^{-1/2} t^{-2} \big\} .
\]
\end{lemma}

\begin{proof}
We begin with the inverse Fourier transform formula
\[
f_i(t) - g_i(t) = \frac1{2\pi} \int_{-\infty}^\infty e^{-ixt} \big( \Phi_i(x) - e^{-V_ix^2/2} \big) \,dx.
\]
On one hand, this integral can be estimated trivially using the first estimate in Lemma~\ref{Phi difference integrals}:
\[
f_i(t) - g_i(t) \ll \int_{-\infty}^\infty \big| \Phi_i(x) - e^{-V_ix^2/2} \big| \,dx \ll V_i^{-3/2}.
\]
On the other hand, we can also integrate by parts twice before estimating, since both $\Phi_i(x)$ and $e^{-V_ix^2/2}$ and their first derivatives tend to~$0$ as~$|x|\to\infty$ (see Remark~\ref{Phi' decays}):
\begin{align*}
f_i(t) - g_i(t) &= \frac 1 {2\pi} \int_{-\infty}^\infty e^{-ixt}(\Phi_i(x) - e^{-V_ix^2/2}) \,dx \\
&= -\frac 1{2\pi t^2} \int_{-\infty}^{\infty} e^{-ixt} \big( \Phi_i''(x) - (V_i^2x^2-V_i) e^{-V_ix^2/2} \big) \,dx
\\
&\ll \frac 1{t^2} \int_{-\infty}^{\infty} \big| \Phi_i''(x) - (V_i^2x^2-V_i) e^{-V_ix^2/2} \big| \,dx \ll \frac {V_i^{-1/2}} {t^2}
\end{align*}
by the second estimate in Lemma~\ref{Phi difference integrals}.
\end{proof}

\begin{lemma} \label{density integral difference}
Assume GRH. For $i\in\{1,2,3\}$ and $q>q_0$, we have
\[
\int_{-\infty}^\infty \big| f_i(t) - g_i(t) \big| \,dt \ll V_i^{-1}.
\]
\end{lemma}

\begin{proof}
By Lemma~\ref{density pointwise difference},
\begin{align*}
\int_{-\infty}^\infty \big| f_i(t) - g_i(t) \big| \, dt &= \int_{|t| \le V_i^{1/2}} \big| f_i(t) - g_i(t) \big| \,dt +\int_{|t| \ge V_i^{1/2}} \big| f_i(t) - g_i(t) \big| \,dt \\
&\ll \int_{|t| \le V_i^{1/2}}  V_i^{-3/2}\,dt + \int_{|t| \ge V_i^{1/2}}\frac {V_i^{-1/2}} {t^2} \,dt \\
&\ll V_i^{-3/2} \cdot V_i^{1/2} + \frac {V_i^{-1/2}} {V_i^{1/2}} \ll  V_i^{-1}. \qedhere
\end{align*}
\end{proof}

We are now ready to compare the probability that the random variables~$X_i$ from Definition~\ref{Xi def}, which are relevant to prime number races, come in a particular order to the probability that normal random variables of the same mean and variance come in a particular order.

\begin{definition} \label{Yi def}
For $i\in\{1,2,3\}$, let $Y_i$ denote a normal variable with mean~$0$ and variance~$V_i$, with the convention that~$Y_1$, $Y_2$, and~$Y_3$ are mutually independent. Note that the density function of~$Y_i$ equals~$g_i(t)$ as defined before Lemma~\ref{density pointwise difference}.
\end{definition}

\begin{theorem} \label{probability difference theorem}
Assume GRH. For any permutation $(i,j,k)$ of $(1,2,3)$ and any $q>q_0$,
\[
\Pr(X_i > X_j > X_k) = \Pr\big(Y_i > Y_j > Y_k\big) + O\bigg( \frac 1 {\phi(q)\log q} \bigg).
\]
\end{theorem}

\begin{proof}
Given the formulas
\begin{align*}
\Pr(X_i > X_j > X_k) &= \iiint\limits_{x>y>z} f_i(x)f_j(y)f_k(z) \,dx\,dy\,dz \\
\Pr(Y_i > Y_j > Y_k) &= \iiint\limits_{x>y>z} g_i(x)g_j(y)g_k(z) \,dx\,dy\,dz,
\end{align*}
we have
\begin{align*}
\big| \Pr &{} (X_i > X_j > X_k) - \Pr(Y_i > Y_j > Y_k) \big| \\
&= \bigg| \iiint\limits_{x>y>z} \big( f_i(x)f_j(y)f_k(z) - g_i(x)g_j(y)g_k(z) \big) \,dx\,dy\,dz \bigg| \\
&= \bigg| \iiint\limits_{x>y>z} \big( f_i(x)-g_i(x) \big) f_j(y)f_k(z) \,dx\,dy\,dz + \iiint\limits_{x>y>z} g_i(x) \big( f_j(y)-g_j(y) \big) f_k(z) \,dx\,dy\,dz \\
&\qquad{}+ \iiint\limits_{x>y>z} g_i(x)g_j(y) \big( f_k(z)-g_k(z) \big) \,dx\,dy\,dz \bigg| \\
&\le \iiint\limits_{\R^3} \big| f_i(x)-g_i(x) \big| f_j(y)f_k(z) \,dx\,dy\,dz + \iiint\limits_{\R^3} g_i(x) \big| f_j(y)-g_j(y) \big| f_k(z) \,dx\,dy\,dz \\
&\qquad{}+ \iiint\limits_{\R^3} g_i(x)g_j(y) \big| f_k(z)-g_k(z) \big| \,dx\,dy\,dz \\
&= \int_\R | f_i(x)-g_i(x) | \,dx + \int_\R | f_j(y)-g_j(y) | \,dy + \int_\R | f_k(z)-g_k(z) | \,dz,
\end{align*}
since each integral of a probability density function over~$\R$ equals~$1$. It follows from Lemma~\ref{density integral difference} and Proposition~\ref{V and eta sizes prop} that
\begin{align*}
\Pr(X_i > X_j > X_k) &= \Pr(Y_i > Y_j > Y_k) + O(V_i^{-1} + V_j^{-1} + V_k^{-1}) \\
&= \Pr(Y_i > Y_j > Y_k) + O\bigg( \frac 1 {\phi(q)\log q} \bigg)
\end{align*}
as desired.
\end{proof}

\section{Proof of the main theorem} \label{finally done section}

By this point we have essentially reduced the problem of asymptotically evaluating the prime-race density $\delta_{q;a_1,a_2,a_3}$ (still under the assumptions from Definition~\ref{a def}) purely to a problem in probability. In this section we complete the proof of Theorem~\ref{main theorem} (including showing how to derive the first assertion from the second) and Corollary~\ref{main cor}, with very little input needed from number theory. We begin with a classical (but perhaps not well known) formula for the probability that three normal variables assume a prescribed ordering.

\begin{lemma} \label{normal arctan lemma}
Let $N_a$, $N_b$, and $N_c$ denote mutually independent random variables with mean~$0$ and variances~$a$, $b$, and~$c$, respectively. Then
\[
\Pr(N_a > N_b > N_c) =  \frac 1 {2\pi} \arctan {\frac {\sqrt{ab+bc+ac}}{b}}.
\]
\end{lemma}

\begin{proof}
If~$Z_1$ and~$Z_2$ are normal random variables with mean~$0$ and correlation coefficient~$\rho$, there is a classical formula (see~\cite[equation~(4)]{RHB} for example) for the ``orthant probability'' that both random variables are positive:
\[
\Pr(Z_1 > 0 \text{ and } Z_2 > 0) = \frac14 + \frac1{2\pi} \arcsin \rho.
\]
We apply this formula with $Z_1=N_a - N_b$ and $Z_2=N_b - N_c$, which are indeed normal random variables with mean~$0$ and correlation coefficient
\begin{align}
\rho &= \frac{\E((N_a - N_b)(N_b - N_c))}{\sqrt{\E((N_a - N_b)^2)\cdot\E((N_b - N_c)^2)}} \label{used independence} \\
&= \frac{-\sigma^2(N_b)}{\sqrt{(\sigma^2(N_a) + \sigma^2(N_b))(\sigma^2(N_b) + \sigma^2(N_c))}} = \frac{-b}{\sqrt{ab+ac+bc+b^2}}, \notag
\end{align}
so that
\begin{align*}
\Pr(N_a > N_b > N_c) &= \Pr(N_a - N_b > 0 \text{ and } N_b - N_c > 0) \\
&= \frac14 -\frac1{2\pi} \arcsin\bigg(\frac b{\sqrt{ab+ac+bc+b^2}} \bigg).
\end{align*}
The lemma now follows from the identities (valid for $0\le x\le y$)
\[
\arcsin\bigg( \frac xy \bigg) = \arctan\bigg( \frac x{\sqrt{y^2-x^2}} \bigg) = \frac\pi2 - \arctan\bigg( \frac{\sqrt{y^2-x^2}}x \bigg). 
\qedhere
\]
\end{proof}

At this point we can complete the proof of an important special case of our main theorem, assuming the restriction from Definition~\ref{a def} that has been in force since that point. Recall the quantity $\delta_{q;a_1,a_2,a_3}$ from Definition~\ref{delta by E^*}.

\medskip\noindent{\em Proof of Theorem~\ref{main theorem}, under the assumption that~$a_1$, $a_2$, and~$a_3$ are either all quadratic residues or all quadratic non\-residues\mod q.}
We may assume that $q>q_0$ from Definition~\ref{q0 def}, since the asymptotic formula is trivially valid for any bounded range of~$q$.
We simply combine the three equalities in Proposition~\ref{E is X prop},
Theorem~\ref{probability difference theorem} (using the notation of Definition~\ref{Yi def}),
and Lemma~\ref{normal arctan lemma}, obtaining
\begin{align*}
\delta_{q;a_1,a_2,a_3} = \Pr(X_1>X_2>X_3) &= \Pr\big(Y_1 > Y_2 > Y_3\big) + O\bigg( \frac 1 {\phi(q)\log q} \bigg) \\
&= \frac 1 {2\pi} \arctan {\frac {\sqrt{V_1V_2+V_2V_3+V_1V_3}}{V_2}} + O\bigg(\frac 1 {\phi(q)\log q}\bigg)
\end{align*}
as claimed.
\qed

It is not difficult to remove the assumption that~$a_1$, $a_2$, and~$a_3$ are either all quadratic residues or all quadratic non\-residues\mod q, at least if we allow ourselves the larger error term asserted in Theorem~\ref{main theorem}. Again all we require is a quick lemma from probability.

\begin{lemma} \label{subevent lemma}
Let~$Z_1$, $Z_2$, and~$Z_3$ be random variables, and let~$\mu_1$, $\mu_2$, and~$\mu_3$ be real numbers. The event
\begin{equation} \label{sym diff event}
\text{exactly one of } (Z_1 > Z_2 > Z_3) \text{ and } (Z_1 + \mu_1 > Z_2 + \mu_2 > Z_3 + \mu_3) \text{ is true}
\end{equation}
is contained in the event
\begin{equation} \label{simpler event}
\big(|Z_1 - Z_2| \le |\mu_1-\mu_2|\big) \text{ or } \big(|Z_2 - Z_3| \le |\mu_2-\mu_3|\big).
\end{equation}
\end{lemma}

\begin{proof}
First observe that
\begin{itemize}
\item if $|Z_1 - Z_2| > |\mu_1-\mu_2|$, then the two inequalities $Z_1 > Z_2$ and $Z_1 + \mu_1 > Z_2 + \mu_2$ are either both true or both false;
\item if $|Z_2 - Z_3| > |\mu_2-\mu_3|$, then the two inequalities $Z_2 > Z_3$ and $Z_2 + \mu_2 > Z_3 + \mu_3$ are either both true or both false.
\end{itemize}
It follows that if both
$
\big(|Z_1 - Z_2| > |\mu_1-\mu_2|\big) \text{ and } \big(|Z_2 - Z_3| > |\mu_2-\mu_3|\big)
$
are true, then
$
(Z_1 > Z_2 > Z_3) \text{ and } (Z_1 + \mu_1 > Z_2 + \mu_2 > Z_3 + \mu_3) \text{ are either both true or both false;}
$
this implication is the contrapositive of the proposition. 
\end{proof}

At this point, we no longer assume that~$a_1$, $a_2$, and~$a_3$ have the same quadratic nature\mod q, although the congruences~\eqref{stw} are still in force.

\begin{proof}[Proof of Theorem~\ref{main theorem} in the general case]
When we do not assume that~$a_1$, $a_2$, and~$a_3$ are either all quadratic residues or all quadratic non\-residues\mod q, we may still use the random variables~$X_i$ and~$Y_i$ from Definitions~\ref{Xi def} and~\ref{Yi def}. However, we cannot rely on full cancellation of the constants in Lemma~\ref{E* decomp lemma}, and so Proposition~\ref{E is X prop} must be modified: the distribution of the vector-valued function $\big( E^{*}(x; q, a_1), E^{*}(x; q, a_2), E^{*}(x; q, a_3) \big)$ is the same as the distribution of the random variable $(\mu_1,\mu_2,\mu_3) + X_{1,2,3}$, where
\[
(\mu_1,\mu_2,\mu_3) = \big( c_q(a_2)+c_q(a_3)-2c_q(a_1) , c_q(a_1)+c_q(a_3)-2c_q(a_2) , c_q(a_1)+c_q(a_2)-2c_q(a_3) \big).
\]
Consequently, the density we want to evaluate now takes the form
\begin{align*}
\delta_{q;a_1,a_2,a_3} &= \Pr(X_1+\mu_1>X_2+\mu_2>X_3+\mu_3) \\
&= \Pr(Y_1+\mu_1>Y_2+\mu_2>Y_3+\mu_3) + O\bigg( \frac1{\phi(q)\log q} \bigg)
\end{align*}
by the proof of Theorem~\ref{probability difference theorem}.
We deduce from Lemma~\ref{subevent lemma} that
\begin{multline*}
\big| \Pr(Y_1+\mu_1>Y_2+\mu_2>Y_3+\mu_3) - \Pr(Y_1>Y_2>Y_3) \big| \\
\le \Pr\big(|Y_1 - Y_2| \le |\mu_1-\mu_2|\big) + \Pr\big(|Y_2 - Y_3| \le |\mu_2-\mu_3|\big).
\end{multline*}
Since the $Y_i$ are mutually independent, $Y_1-Y_2$ is a normal random variable with variance $V_1+V_2$, and hence its density function is bounded pointwise by the constant $1/\sqrt{2\pi(V_1+V_2)}$; the analogous bound applies to the density function of $Y_2-Y_3$. In particular,
\begin{multline*}
\big| \Pr(Y_1+\mu_1>Y_2+\mu_2>Y_3+\mu_3) - \Pr(Y_1>Y_2>Y_3) \big| \\
\le \frac{2|\mu_1-\mu_2|}{\sqrt{2\pi(V_1+V_2)}} + \frac{2|\mu_2-\mu_3|}{\sqrt{2\pi(V_2+V_3)}}.
\end{multline*}
However, it is a standard fact that $c_q(a) \ll_\ep q^\ep$ (see~\cite[Definitions~1.2 and~2.4]{FiM}). Since each $V_i \gg 1/\sqrt{\phi(q)\log q}$ by Proposition~\ref{V and eta sizes prop}, we deduce that
\begin{equation*}
\big| \Pr(Y_1+\mu_1>Y_2+\mu_2>Y_3+\mu_3) - \Pr(Y_1>Y_2>Y_3) \big| \ll_\ep \frac{q^\ep}{\sqrt{\phi(q)\log q}} \ll_\ep q^{-1/2+\ep}.
\end{equation*}
In conclusion,
\begin{align*}
\delta_{q;a_1,a_2,a_3} &= \Pr(Y_1+\mu_1>Y_2+\mu_2>Y_3+\mu_3) + O\bigg( \frac1{\phi(q)\log q} \bigg) \\
&= \Pr(Y_1>Y_2>Y_3) + O_\ep\bigg( q^{-1/2+\ep} + \frac1{\phi(q)\log q} \bigg) \\
&= \frac1{2\pi} \arctan \frac{\sqrt{V_1V_2+V_1V_3+V_2V_3}}{V_2} + O_\ep(q^{-1/2+\ep})
\end{align*}
by Lemma~\ref{normal arctan lemma}, as claimed.
\end{proof}

\begin{proof}[Proof of Corollary~\ref{main cor}]
Since $V_i = 4V(q) (1+\eta_i)$ by Definition~\ref{V def}, we may restate Theorem~\ref{main theorem} as
\begin{multline*}
\Pr(X_i > X_j > X_k) \\
= \frac 1 {2\pi} \arctan {\frac {\sqrt{  (1+\eta_i)  (1+\eta_j)+  (1+\eta_i) (1+\eta_k)+ (1+\eta_j) (1+\eta_k)}}{ (1+\eta_j)}} + O_\ep(q^{-1/2+\ep}).
\end{multline*}
It is an easy calculus exercise to compute the linear approximation at the origin to the twice-differentiable function above, obtaining
\begin{align*}
\Pr(X_i > X_j > X_k) &=  \frac16 + \frac{\eta_i}{8\pi\sqrt3} - \frac{\eta_j}{4\pi\sqrt3} + \frac{\eta_k}{8\pi\sqrt3} + O(\eta_i^2 + \eta_j^2+ \eta_k^2) + O_\ep(q^{-1/2+\ep}).
\end{align*}
The corollary now follows from the estimate $\eta_i \ll (\log\log q)/{\log q}$ given in Proposition~\ref{V and eta sizes prop}.
\end{proof}

It turns out that while it is possible for the $\eta_i$ to be as large as $\Omega((\log\log q)/\log q)$, they are usually rather smaller; in such a situation, the error term given in Corollary~\ref{main cor} can be reduced considerably. Indeed, an asymptotic formula for $\delta_{q;a_1,a_2,a_3}$ was given by Lamzouri in a form where the secondary main terms and error terms had a more explicit dependence on arithmetic quantities like the $V_i$, including the one we define now.

\begin{definition} \label{B def}
For any reduced residue classes~$a$ and~$b$ modulo~$q$, define
\begin{align*}
B_q(a,b) &= \sum_{\substack{\chi \mod q \\ \chi \ne \chi_0}} \big( \chi(ab^{-1}) + \chi(ba^{-1}) \big) b_+(\chi).
\end{align*}
(Note for example that $B_q(a,a)=V(q)$ from Definition~\ref{V def}.)
\end{definition}

The following result, which is~\cite[Corollary~2.3]{lam} translated into our notation, applies to {\em any} distinct reduced residues~$a_1$, $a_2$, and~$a_3\mod q$.

\begin{theorem}[Lamzouri]
We have
\begin{multline} \label{lam result}
\delta_{q;a_1,a_2,a_3} = \frac 1 6 + \frac 1 {4 \sqrt\pi} \frac{c_q(a_3) - c_q(a_1)}{\sqrt{V(q)}} + \frac 1 {4 \pi \sqrt{3}} \frac{B_q(a_1,a_2)+ B_q(a_2,a_3) - 2B_q(a_1,a_3)}{V(q)} \\
+ O \bigg(\frac {c_q(1)^2}{V(q)} +\frac {\max_{1\le i<j\le 3}|B_q(a_i,a_j)|^2}{V(q)^2} \bigg).
\end{multline}
\end{theorem}

One motivation for stating this result is to calculate the secondary main terms in our special case $a_1^2\equiv a_2^2\equiv a_3^2\mod q$ and $c_q(a_1)=c_q(a_2)=c_q(a_3)$ in the notation of equation~\eqref{cq def}. Under these assumptions,
\begin{align*}
B_q(a_i,a_j) &= 2 \bigg( \sum_{\substack{\chi \mod q \\ \chi(a_i) = \chi(a_j)}} b_+(\chi) - \sum_{\substack{\chi \mod q \\ \chi(a_i) = -\chi(a_j)}} b_+(\chi) - b_+(\chi_0) \bigg) = \frac{V_0 + V_k - V_i - V_j}{16} - 2b_+(\chi_0)
\end{align*}
in the notation of Definition~\ref{V def} (and thus, from Proposition~\ref{V and eta sizes prop}, we have $B_q(a_i,a_j)^2/V(q)^2 \ll (\log\log q)^2/(\log q)^2$). In particular,
\[
B_q(a_1,a_2)+ B_q(a_2,a_3) - 2B_q(a_1,a_3) = \frac{V_1-2V_2+V_3}8.
\]
Using this identity in equation~\eqref{lam result}, and making the change of variables $V_i = 4V(q)(\eta_i+1)$ from Definition~\ref{V def}, reveals that the secondary main terms in~\eqref{lam result} are exactly equal to those in Corollary~\ref{main cor}. (As we see, Lamzouri's result gives yet another secondary main term in the case where $c_q(a_1)\ne c_q(a_3)$, while our method simply gives an error term of that order of magnitude.)

\section{Discussion} \label{discussion section}

We have already discussed, at the end of Section~\ref{almost unanimous section}, the algebraic aspects of our special three-way races and the prospects for generalizing our method in that regard. In this final section we make some additional remarks about the analytic aspects of this paper, under the continuing assumptions of GRH and~LI.

The moral we hope to emphasize is that, {\em to give asymptotic formulas for prime number race densities $\delta_{q;a_1,\dots,a_r}$ from Definition~\ref{delta by pi} whose error terms are small, one should always utilize a main term that is an ``ordering probability'' for a multivariate normal distribution}. If $Z=(Z_1,\dots,Z_r)$ is a normal random variable in~$\R^r$ whose covariance matrix is identical to the matrix of covariances $B_q(a_i,a_j)$ (from Definition~\ref{B def}) corresponding to our prime number race, then the difference between $\delta_{q;a_1,\dots,a_r}$ and $\Pr(Z_1>\cdots>Z_r)$ will decay like a negative power of~$q$, as in Theorem~\ref{main theorem}. Therefore it is best, we claim, to primarily evaluate $\delta_{q;a_1,\dots,a_r}$ as an ordering probability of this type; further asymptotic evaluations can then be made for the ordering probability itself, an object that is purely probabilistic.

This idea is certainly well established in this topic on a heuristic level, being motivated, for example, by the central limit results established by Rubinstein and Sarnak~\cite{RS}. Using an ordering probability as the main term is implicit in the work of Lamzouri~\cite{lamother,lam} (though the comparisons there were carried out mainly on the characteristic function side) and more explicitly in the work of Harper and Lamzouri~\cite[Section~4.1]{HL}.

While we wanted to describe the construction of our atypical normalizations $E^*(x;q,a)$ of the error terms in prime counting functions for arithmetic progressions, and how they resulted in independent random variables that allowed for a much more elementary analysis, that independence is not necessary to our moral, as it happens. The multivariate normal random variable~$Z$ alluded to above is permitted to have correlations among the coordinates, and even to have nonzero means in each coordinate (arising, in our case, from Chebyshev's bias against quadratic residues). Because the matrix of covariances $B_q(a_i,a_j)$ turns out to be close to $V(q)$ (from Definition~\ref{V def}) times the identity matrix, and the vector of means turns out to be small compared to $\sqrt{V(q)}$, it is possible to produce asymptotics for these ordering probabilities in terms of the means and covariances, as was done in~\cite{lam}.

Moreover, in the case where the residues~$a_j$ are all quadratic residues or all quadratic nonresidues (so that the means of the individual limiting distributions are all equal), these ordering probabilities can actually be evaluated in closed form for $r\le 4$, because they are equivalent (by the method of proof of Lemma~\ref{normal arctan lemma}) to orthant probabilities in at most three dimensions. (The cases $r\le2$ are trivial because zero- and one-dimensional orthant probabilities are trivial under the assumption of equal means.) Recalling Definitions~\ref{b+ def}, \ref{V def}, and~\ref{B def}, we define the further notation
\begin{align*}
\rho_{12} &= \frac{-V(q)+B_q(a_1,a_2)-B_q(a_1,a_3)+B_q(a_2,a_3)}{2\sqrt{(V(q) - B_q(a_1,a_2))(V(q) - B_q(a_2,a_3))}} \\
\rho_{13} &= \frac{B_q(a_1,a_3)-B_q(a_2,a_3)-B_q(a_1,a_4)+B_q(a_2,a_4)}{2\sqrt{(V(q) - B_q(a_1,a_2))(V(q) - B_q(a_3,a_4))}} \\
\rho_{23} &= \frac{-V(q)+B_q(a_2,a_3)-B_q(a_2,a_4)+B_q(a_3,a_4)}{2\sqrt{(V(q) - B_q(a_2,a_3))(V(q) - B_q(a_3,a_4))}}
\end{align*}
Then for {\em any} distinct reduced residues~$a_1$, $a_2$, $a_3$, and~$a_4\mod q$, known formulas for orthant probabilities~\cite[equations~(4) and~(5)]{RHB} imply that the asymptotic formulas
\begin{align*}
\delta_{q;a_1,a_2,a_3} &\sim \frac14 + \frac1{2\pi} \arcsin \rho_{12} \\
\delta_{q;a_1,a_2,a_3,a_4} &\sim \frac18 + \frac1{4\pi} \big( \arcsin \rho_{12} + \arcsin \rho_{13} + \arcsin \rho_{23} \big)
\end{align*}
hold up to a negative power of~$q$. As a reality check, if we exploit the fact that the $B_q(a_i,a_j)$ are negligible in size compared to $V(q)$, then $\rho_{12}\sim\rho_{23}\sim-\frac12$ and $\rho_{13}=o(1)$, which leads to the evaluations $\delta_{q;a_1,a_2,a_3} \sim \frac16$ and $\delta_{q;a_1,a_2,a_3,a_4} \sim \frac1{24}$ as expected from the central limit theorems of~\cite{RS}.

In this paradigm, we have approximated our number-theoretic limiting logarithmic distributions by normal distributions with the same mean and variance. Of course, the higher (even central) moments of the two distributions will not match in general, which is a source of error when passing from one measure to the other. It is worth mentioning that for two-way races, asymptotic formulas for the densities exist~\cite[Theorem~1.1]{FiM} that incorporate the contribution from higher moments and, correspondingly, have error terms that can be made as small as an arbitrary power of~$q$. It would be an interesting project to attempt to produce analogous formulas for prime number races with three or more contestants.

Nevertheless, we hope the viewpoint that prime race densities are best approximated explicitly by ordering probabilities for multivariate normal random variables has some illuminating benefit to practitioners of comparative prime number theory.

\section*{Acknowledgments}

The authors thank Adam Harper and Youness Lamzouri for very helpful conversations related to the topic of this paper. The second author was supported by a Natural Sciences and Engineering Research Council of Canada Discovery Grant.

\bibliographystyle{amsplain}
\bibliography{DCTWPNR}

\end{document}